\documentclass[12pt,reqno]{amsart}

\usepackage[T1]{fontenc}
\usepackage{enumitem}
\usepackage{hyperref}
\hypersetup{
  colorlinks   = true,
  citecolor    = blue,
  linkcolor    = blue
}

\usepackage{amsmath,amsfonts,amsthm,amssymb,color,tikz, comment, xcolor}
\usepackage{mathrsfs}   
\usepackage[normalem]{ulem}

\usepackage{pdfsync}
\usepackage[font={scriptsize}]{caption}
\usepackage{environ}

\usepackage[left=1in, right=1in, top=1.1in,bottom=1.1in]{geometry}
\makeatletter
\@namedef{subjclassname@2020}{%
  \textup{2020} Mathematics Subject Classification}

\newcommand{\der}{\delta}

\newcommand{\pl}{\partial}

\newcommand{\bx}{{\bf x}}

\usepackage{accents}

\newcommand{\E}{\mathbb E}
\newcommand{\R}{\mathbb R}
\newcommand{\N}{\mathbb N}
\newcommand{\PP}{\mathbb P}

\newcommand{\bd}{\mathbf}
      
\newcommand{\ca}{\mathcal A}
\newcommand{\cb}{\mathcal B}
\newcommand{\cac}{\mathcal C}

\newcommand{\cf}{\mathcal F}

\newcommand{\ch}{\mathcal H}
\newcommand{\ci}{\mathcal I}

\newcommand{\cl}{\mathcal L}
\newcommand{\cm}{\mathcal M}
\newcommand{\cn}{\mathcal N}
\newcommand{\co}{\mathcal O}
\newcommand{\cp}{\mathcal P}

\newcommand{\cs}{\mathcal S}

\newcommand{\cu}{\mathcal U}

\newcommand{\al}{\alpha}

\newcommand{\ep}{\varepsilon}

\newcommand{\ga}{\gamma}

\newcommand{\ka}{\kappa}
\newcommand{\la}{\lambda}

\newcommand{\oom}{\Omega}

\newcommand{\si}{\sigma}

\newcommand{\vp}{\varphi}

\newcommand{\lp}{\left(}
\newcommand{\rp}{\right)}
\newcommand{\lc}{\left[}
\newcommand{\rc}{\right]}
\newcommand{\lcl}{\left\{}
\newcommand{\rcl}{\right\}}
\newcommand{\lln}{\left|}
\newcommand{\rrn}{\right|}
\newcommand{\lla}{\left\langle}
\newcommand{\rra}{\right\rangle}

\numberwithin{equation}{section}
\newtheorem{theorem}{Theorem}[section]
\newtheorem{lemma}{Lemma}[section]
\newtheorem{assumption}{Assumption}[section]
\newtheorem{proposition}{Proposition}[section]

\newtheorem{definition}{Definition}[section]

\newtheorem{notation}{Notation}[section]

\theoremstyle{remark}
\newtheorem{remark}{Remark}[section]

\newcommand{\bean}{\begin{eqnarray*}}
\newcommand{\eean}{\end{eqnarray*}}
\newcommand{\ben}{\begin{enumerate}}
\newcommand{\een}{\end{enumerate}}
\newcommand{\beq}{\begin{equation}}
\newcommand{\eeq}{\end{equation}}

\title[Mean Field Control with Regime Switching]{Mean field control and finite agent approximation for regime-switching jump diffusions}

\keywords{viscosity solutions, mean field control, master equation, propagation of chaos}

\subjclass[2020]{35D40, 35F21, 35Q91, 49L25}

\author[E. Bayraktar]{Erhan Bayraktar}

\address[E. Bayraktar]{Department of Mathematics\\
University of Michigan\\
\newline Ann Arbor, MI 48109\\
United States
}
\email{erhan@umich.edu}

\author[A. Cecchin]{Alekos Cecchin}
\address[A. Cecchin]{Department of Mathematics ``Tullio Levi-Civita''\\
University of Padova\\
\newline Padova, 35121\\
Italy
}
\email{alekos.cecchin@unipd.it}

\author[P. Chakraborty]{Prakash Chakraborty}
\address[P. Chakraborty]{ Harold and Inge Marcus Department of Industrial and Manufacturing Engineering \\
 The Pennsylvania State University \\
University Park, PA 16802
United States
}
\email{prakashc@psu.edu}

\date{\today}

\thanks{E. Bayraktar is supported in part by the National Science Foundation under grant DMS-2106556 and in part by the Susan M. Smith Professorship.}

\begin{document}

\begin{abstract}
We consider a jump-diffusion mean field control problem with regime switching in the state dynamics. The corresponding value function is characterized as the unique viscosity solution of a HJB master equation on the space of probability measures. Using this characterization, we prove that the value function, which is not regular, is the limit of a finite agent centralized optimal control problem as the number of agents go to infinity, with an explicit convergence rate. 
Assuming in addition  that the value function is smooth, we establish a quantitative propagation of chaos result for the optimal trajectory of agent states.
\end{abstract}

\maketitle

\section{Introduction}\label{sec:intro}
Mean field control problems with regime switching state dynamics is a growing activity of research starting from the works of \cite{gen-yin, wang-zhang-12, wang-zhang-17, xiong} and most was recently considered by \cite{song}. This line of research activity is motivated from the need to control large identical interacting control systems where the interaction is mean-field and there exists a common noise affecting each system. Though Brownian common noise has been studied at length in the literature, a simpler example is a Markov chain representing the regime switches experienced by each system in the network. From a modeling perspective too such hybrid models have been shown to be extremely powerful \cite{xiong-9, xiong-7, xiong-19, zhou-yin}. In contrast to a common Markov switching, it is indeed possible to have a network where each system has its own switching mechanism. Such a model has been analyzed in \cite{wang-zhang-12, xiong}. However since there is already a mean field interaction between the systems it is sensible to have just one common random switching mechanism to incorporate the hybrid nature of the network.

In this work we model each individual system identically using a controlled jump-diffusion, where the jumps are independent and identically distributed (i.i.d.) and the drift, diffusion and jump intensity coefficients are modulated by the same Markov chain. The mean-field limit of this network is expected to be characterized by the distribution, conditioned on the history of the switching process, of the solution of a McKean Vlasov jump diffusion with regime switching coefficients. This mean field limit in an uncontrolled setup and without jumps has been analyzed in \cite{george-yin}. In some papers including \cite{gen-yin-2, gen-yin, xiong} the mean field interaction has been represented by the conditional expectation of the solution of the McKean Vlasov jump diffusion. In \cite{george-yin, xiong}, the stochastic maximum principle has been studied and obtained. In contrast, we use Bellman's dynamic programming principle approach to obtain a master equation as in \cite{bayraktar-cosso-pham-2} and \cite{pham-wei}. This equation can be studied by Lions' lifting of measure-dependent functions to appropriate functions on the space of random variables. Traditionally the Wasserstein space of measures and the $\cl^2$ space of random variables are used for this purpose. If now the solution is regular enough a classical interpretation of the master equation is used. Else, one relies on appropriate viscosity solution theory. Because of the jump terms present in our model, working with Lion's derivative is not enough to state the master equation in our case, and one needs to rely on the linear functional derivative alone. This has been mentioned and explored in \cite{soner} while analyzing a controlled McKean Vlasov jump diffusion without regime switches. In addition \cite{soner} provides a theory of viscosity solutions on the space of probability measures that does not rely on lifting to a Hilbert space. It is this notion of viscosity solution that we will use, modified of course to account for the regime switches, in addition to the metric (equivalent to weak topology) on the space of probability measures defined therein. 

The mean field control problem we analyze in the paper is the same as in \cite{soner}, and is slightly different from the usual one. Namely, in the finite-agent centralized optimization, the strategy of any agent is usually expected to depend on the private state and on the empirical measure of the system, while here it is independent of the private state. In other words, each agent uses as control the same stochastic process which is given by the central planner, and we believe that there are several situations which can be modeled in this way. As a consequence, the controls employed in the limiting mean field control problem are just deterministic functions of time.

In order to prove convergence of the finite state optimization to the mean field optimization, in the usual setup, there are several approaches. Let us mention \cite{djete-extended, djete-possamai-tan, lacker} where the problem is tackled probabilistically through compactness arguments. In addition convergence with rates can be obtained through FBSDE techniques, see \cite{carmona-delarue} and \cite{lauriere2022convergence}, the latter with interaction also through the law of the control; but these papers assume convexity in the measure argument. Another recent paper \cite{germain2022rate} establishes a convergence rate, assuming that the limit value function is smooth (which should hold e.g. under convexity).

In this present work, we prove convergence of the value functions with a convergence rate without assuming convexity, nor that the value function is smooth. Instead, we make use of the viscosity solution characterization of the optimal value function, which is in general not differentiable.  
 This is most related to the recent papers \cite{cecchin} and \cite{gangbo-mayorga-swiech}. In \cite{gangbo-mayorga-swiech}, the viscosity solutions of HJB equations in finite-agent deterministic or stochastic optimal control problems are shown to converge to that of a limiting HJB equation in the space of probability measures. The latter equation is interpreted through Lions' lifting in the $\cl^2$ sense. However convergence rates are absent in \cite{gangbo-mayorga-swiech}. But since we rely on our particular viscosity solution structure, we can  adopt the ideas in \cite{cecchin} even though the problem addressed there is in the space of probability measures with finite support. We employ a doubling of variables argument, using the distance-like function introduced in \cite{soner}.


A recent preprint \cite{carda-souga-1} provides a rate for the convergence of the value function, under general assumptions. As explained above, the mean field control problem we consider here is slightly different and thus the proof of convergence is completely different. As a consequence, we obtain a better rate than in \cite{carda-souga-1} (see Remark \ref{remark:carda-souga}). We mention, however, that we impose a structural assumption on the coefficients, in order to apply the theory of \cite{soner}. 

We also provide, for completeness,  a quantitative propagation of chaos for the convergence of the optimal trajectories, but, to obtain this result, we suppose that the limit value function is smooth, as in \cite{germain2022rate}. We mention, in this respect, a recent preprint \cite{carda-souga-2} where the authors establish a quantitative propagation of chaos, without assuming regularity of the limit value function: they show that the value function is $C^1$ in an open and dense subset of the space of probability measures and thus prove convergence if the initial distribution belongs to that set.

The rest of the paper is organized as follows. In Section~\ref{sec:intro}, we mention the mean field control problem as motivation to introduce associated assumptions, the state space under consideration, along with the particular metric on $\cp(\R)$ borrowed from \cite{soner} used in our paper. These notions will be used throughout the rest of this article. In Section~\ref{sec:MFC}, we study the mean field control problem in detail. In particular, we obtain an HJB equation, define a viscosity solution theory suitable for our purposes and prove that the value function is the unique viscosity solution to that HJB. Next in Section~\ref{sec:finite-agent}, we introduce the finite agent centralized control problem and prove the main convergence results: we  show using viscosity solutions that the value function in the finite agent control problem is uniformly approximated by that in the mean field control problem. Finally, in Section~\ref{sec:poc} we prove under additional assumptions, a propagation of chaos result showing that the optimal trajectories are uniformly close as well.
 
\section{Preliminaries}\label{sec:prelim}
\subsection{The control problem}
Consider a complete probability space $(\oom, \cf, ({\cf}_s)_{s \in [0,T]}, \PP)$ on which is defined a Brownian motion $(W_s)_{s \in [0,T]}$, where $T>0$ is an arbitrary fixed time horizon. Let $\al$ be a continuous time Markov chain with finite state space $\cs=\{1,\ldots,s_0\}$ and generator $Q = (q_{ij})_{1\leq i,j\leq s_0}$. Let $\cp(\R)$ denote the class of probability measures on $\R$. For $(t,\rho) \in [0,T] \times \mathcal{P}(\R)$ we consider the following controlled McKean-Vlasov stochastic differential equation with initial condition $\cl(X_t\vert \cf_{t-}^{\al}) = \rho$ and $\al_t=i_0 \in \cs$:
\beq\label{eq:SDE}
dX_s = b(s, X_s, \mu_s, v_s, \al_{s-})ds + \si(s, X_s, \mu_s, v_s, \al_{s-}) dB_s + dJ_s, \qquad t \leq s \leq T,
\eeq
where $\mu_s = \cl(X_s \vert \cf_{s-}^{\al})$, $\cf^{\al}$ is the filtration generated by the Markov chain $\al$, $J_s$ is a purely discontinuous process with controlled intensity $\la(s, X_s, \mu_s, v_s, \al_{s-})$ and the jump sizes are i.i.d. from some distribution $\gamma \in \mathcal{P}(\R)$. {Furthermore, $v_s:= v(s, \mu_s, \al_{s-})$ is a deterministic feedback control of time $s$, conditional law $\mu_s$ and regime state $\al_{s-}$ taking values in a prescribed {Polish space} $A$. Denote $\ca$ to be this class of admissible controls.} The solution to \eqref{eq:SDE} depends on $t$, $\rho$, $i_0$ and $v$. However for ease in presentation we will sometimes omit these and the solution will be denoted just by $X^v$ or just $X$, when the dependence on the initial conditions and the control is clear from the context. 

We then consider the value function
\beq\label{eq:value}
V(t, \rho, i_0)
:= \inf_{v \in \ca} \E \lc \int_t^T f(s, X_s, \mu_s, v_s, \al_{s-}) ds + h(T, X_T, \mu_T, \al_T) \rc,
\eeq
for given functions $f$ and $h$, and where $v_s \equiv v(s, )$

\begin{assumption}
The space of probability measures is endowed with the weak* topology $\si(\cp(\R), \cac_b(\R))$ where $\cac_b(\R)$ is the space of continuous and bounded functions in $\R$. The weak* convergence $\mu_n \to \mu$ is equivalent to: $\langle \mu_n, f \rangle$ converges to $\langle \mu, f \rangle$ for every $f \in \cac_b(\R)$. Whenever required we can use any of the metrics $d(\mu, \mu')$ on $\cp(\R)$ which permits it to be topologically equivalent to the weak* topology. In particular, in the sequel we will use the metric $d(\mu, \mu') = \sum_j c_j |\langle \mu - \nu, f_j \rangle|$, when restricted to a suitable compact set $\mathscr{O} \subset \cp(\R)$, where $\{f_j\}_j$ is a class of polynomials containing all monomials (see later for more details). If needed, $\cs$ is endowed with the metric $d_{\cs}$ satisfying $d_{\cs}(i_0, j_0) = \mathbf{1}_{\{i_o \neq j_0\}}$.
\end{assumption}

\begin{assumption}\label{hyp:main}
There exist constants $C_0$
\begin{enumerate}
\item[(i)] For any $s \in [0,T]$, $x \in \R$, $\mu \in \cp(\R)$, $v \in A$, $i_0 \in \cs$,
$$
|b(s,\mu,v,i_0)| + |\si(s,\mu,v,i_0)| + |\la(s,\mu,v,i_0)| \leq C_0
$$

\item[(ii)] There exists a finite set $\ci \in \N$ such that for any $\mu, \mu' \in \cp(\R)$, $t,s \in [0,T]$, $x, y \in \R$, $i_0 \in \cs$
\begin{multline*}
|b(t,\mu,v,i_0) - b(s,\mu',v,i_0)| + |\si(t,\mu,v,i_0) - \si(s,\mu',v,i_0)| \\ + |\la(t,\mu,v,i_0) - \la(s,\mu',v,i_0)|   \leq \ka_0 \lp |t-s|  + \sum_{k \in \ci} \lvert \lla \mu-\mu', x^k \rra \rvert \rp.
\end{multline*}

\item[(iii)] $\ga$ has $\delta$-exponential moment:
$$
\int_{\R} \exp(\delta |x|) \ga(dx) < \infty. 
$$

\item[(iv)] $f$ and $h$ are bounded. Furthermore there exists a finite set $\ci \in \N$ such that for any $\mu, \mu' \in \cp(\R)$, $t,s \in [0,T]$ and $x, y \in \R$:
\begin{multline*}
\lln f(t, x, \mu, v, i_0) - f(s, y, \mu', v, i_0) \rrn + \lln h(t, x, \mu, i_0) - h(s, y, \mu', i_0) \rrn\\
\leq \ka_1 \lp |t-s| + |x-y| + \sum_{k \in \ci} \lln \lla \mu - \mu', x^k \rra \rrn \rp.
\end{multline*}

\item[(v)] All the functions $b, \sigma, \lambda, f, h$ are Lipschitz-continuous in the measure argument for the 2-Wasserstein distance $W_2$.


\end{enumerate}

\end{assumption}

\begin{remark}
We should note that although the 2-Wasserstein distance and the distance $d$ are topologically equivalent, they are not strongly equivalent. We need both of (ii) and (v) above for our main result, Theorem~\ref{thm:1}. See e.g. Propositions~\ref{prop:Lipsct} and \ref{prop:Was-Lip}.
\end{remark}

\subsection{State space}
\begin{notation}
Since the Brownian motion has exponential moments, the solution to \eqref{eq:SDE} also has exponential moments owing to boundedness of the coefficients in Assumption~\ref{hyp:main}.  
As in \cite{soner} we also consider the optimal control problem in $\co \times \cs$,  where $\co := [0,T) \times \cm$ and $\cm$  is the set of probability measures with $\der$-exponential moments. Denote $\overline{\co} := [0,T] \times \cm$. Let 
\beq\label{eq:e^der}
e_{\der} (x) := \exp \lp \der \lc \sqrt{x^2+1} -1 \rc \rp, \quad x \in \R,
\eeq
where $\der$ is as in Assumption~\ref{hyp:main}. For $b > 0$ we denote
$$
\cm_b := \lcl \mu \in \cp(\R) : \lla \mu, e_{\der} \rra \leq b \rcl.
$$
For $M \in \N$, 
$$
\co_M = \lcl (t, \mu) \in [0,T)\times \cp(\R): \lla \mu, e_{\der} \rra \leq {M e^{K^{\ast}t}} \rcl,
$$
where 
$$
K^{\ast} = \dfrac{\der C_0}{2} \lp 2+ C_0 + \der C_0 \rp + C_0 \lp \int_{\R} e^{\der |x|} \ga(dx) - 1 \rp.
$$
Also denote
$$
\overline{\co}_M = \lcl (t, \mu) \in [0,T] \times \cp(\R): \lla \mu, e_{\der} \rra \leq {M e^{K^{\ast}t}} \rcl,
$$
and $\overline{\co} = \cup_{M=1}^{\infty} \overline{\co}_M$.
\end{notation}

\begin{remark}\label{rem:mu=nu}
If $\mu, \nu \in \cm$, then
$$
\mu = \nu \qquad \Leftrightarrow \qquad \lla \mu - \nu, x^k \rra = 0 \quad \forall k \in \N.
$$
\end{remark}
\begin{definition}
A function  $\vp:\cp(\R) \to \R$ is said to have a linear functional derivative at $\mu \in \cp(\R)$ if there exists a function $D_m \vp: \cp(\R) \times \R \to \R$ such that for every $\mu, \mu' \in \cp(\R)$ the following relation holds
$$
\vp(\mu) - \vp(\mu') = \int_0^1 \int_{\R} D_m \vp(r \mu + (1-r)\mu', x) (\mu - \mu') (dx) dr.
$$
The function $D_{m^2}^2 \vp: \cp(\R) \times \R \times \R \to \R$ stands for the second linear functional derivative of $\vp$ at $\mu$ and is defined as the linear derivative of $D_m \vp$.
\end{definition}

\begin{remark}
Consider the linear function $\vp(\mu) = \lla \mu, f \rra$ for some $f : \R \to \R$. Then $D_m \vp(\mu, x) = f(x)$ for any $(\mu, x) \in \cp(\R) \times \R$. 
\end{remark}

In the following we provide a few details on the particular distance $d(\mu, \nu)$ we consider in this article. We consider a slight variation of the distance introduced in \cite{soner}, to which we refer for more details. We start off with a set of polynomials against which we will integrate our measures.
\begin{definition}\label{def:Theta}
Let $\Theta$ be the minimal set of polynomials such that
\begin{enumerate}
\item[(i)] for any $g \in \Theta$, $g^{(k)} \in \Theta$ for all $k=0, \ldots, \textnormal{deg}(g)$;
\item[(ii)] for any $g \in \Theta$, $\sum_{k=1}^{\textnormal{deg}(g)}m_k g^{(k)} \in \Theta$ where $m_k = \frac{1}{k!} \int_{\R} y^{k} \ga(dy)$; 
\item[(iii)] for any $g\in\Theta$, $(g')^2 \in\Theta$. 
\item[(iv)] all monomials $\{x^k\}_{k=1}^{\infty}$ is contained in $\Theta$. 
\end{enumerate}
\end{definition}
It can be shown that $\Theta$ is countable. Let $\{f_j\}_{j=1}^{\infty}$ be an enumeration of $\Theta$. Fix a $b > 0$ and recall $\cm_b$ defined earlier. 
For every $f_j \in \Theta$, consider the finite index set $I_j$ so that the set 
$\{f_i : i\in I_j\}$ is the set of polynomials obtained from $f_j$ by one of the first three above operations. Since measures $\cm_b$ have bounded exponential moments, 
\[
s_j (b) := 1+ \sup_{\mu\in \cm b} \langle \mu , f_j\rangle ^2  < \infty \qquad \forall j, 
\] 
and thus set 
\beq 
c_j(b) := \left( \sum_{k\in I_j} 2^k \right)^{-1} 
\left( \sum_{k\in I_j} s_k(b) \right)^{-2}.
\eeq

Then we have the following result the proof of which can be found in \cite{soner}.
\begin{lemma}\label{lem:metric-d}
With the above choice $\{ c_j(b) \}_{j=1}^{\infty}$
$$
d(\mu, \nu; b) = \sum_{j=1}^{\infty} c_j(b) |\lla \mu- \nu, f_j \rra|,
$$
defines a metric on $\cm_b$. Furthermore a sequence $\{\mu_n\}_{n \in \N}$ in $\cm_b$ converges weakly to $\mu \in \cm_b$ if and only if $\lim_{n \to \infty} d(\mu_n, \mu; b) =0$.
\end{lemma}
As a consequence of the definition, we have the following facts, which will be useful in the sequel. Their proofs can be found in \cite{soner}.
\begin{enumerate}
\item $c_j(b) \leq 2^{-j}$ and consequently $\sum_j c_j(b) \leq 1$;
\item for each $f_i \in \Theta$ with $i\in I_j$,  
$c_j(b) \leq c_i(b)$. 
\item $\sum_j c_j(b) {\lla \mu, f_j \rra}^2 \leq 1$, for all $\mu \in \cm_b$.
\end{enumerate}
\begin{remark}\label{rem:metric-d}
In the following we will fix an $M \in \N$ and run our analyses on $\co_M$. To that effect we fix $b=Me^{K^{\ast}T}$ and omit the dependence of $c_j$ and $d$ on $b$.
\end{remark}

\begin{notation}
Associated with each pair of states $(i_0, j_0) \in \cs \times \cs$, $i_0 \neq j_0$ of the Markov chain $\al$ we denote
$$
[M_{i_0 j_0}^{\al}](t) = \sum_{0 \leq s \leq t} \mathbf{1}_{\{ \al_{s-} = i_0 \}}\mathbf{1}_{\{ \al_{s} = j_0 \}}, 
\quad
\lla M_{i_0 j_0}^{\al} \rra (t) = \int_0^t q_{i_0 j_0} \mathbf{1}_{\{ \al_{s-} = i_0 \}} ds,
$$ 
Finally for $t \in [0,T]$ the process $M_{i_0 j_0}^{\al}(t)$ is defined by
$$
M_{i_0 j_0}^{\al}(t) = [M_{i_0 j_0}^{\al}](t) - \lla M_{i_0 j_0}^{\al} \rra (t).
$$ 
It can be shown that this is a purely discontinuous and square integrable martingale with respect to the complete filtration $\cf_t$. See e.g. \cite{george-yin}. 
\end{notation}
As necessary we will use the following notation to measure distance between states in the state space $\cs$ of the Markov chain $\al$. 
\begin{notation}
For every $i_0, j_0 \in \cs$ 
$\tilde d(i_0, j_0) = \mathbf{1}_{\{i_0 \neq j_0\}}$.
\end{notation}

\section{Mean field control problem}\label{sec:MFC}
\subsection{Dynamic Programming}
\begin{lemma}
The following dynamic programming principle holds:
\beq\label{eq:DPP}
V(t,\rho,i_0) = \inf_{v \in \ca} \E \lc \int_t^{\theta} f(s, X_s, \mu_s, v_s, \al_{s-}) ds + V(\theta, \mu_{\theta}, \al_{\theta}) \rc, \quad \forall \theta \in [t, T],
\eeq
\end{lemma}
\begin{proof}
Since the McKean-Vlasov control problem considered here is deterministic, the dynamic programming principle follows from classical results. See, for instance \cite{fleming-soner}. 
\end{proof}
\begin{notation}
For a given class of functions ${u} = (u(t,\mu,i_0,x))_{i_0 \in \cs}$, the operator $\cl_t^{\mu,i_0,v}$ acts on the $x$ and $i_0$ variables and is given by
\begin{multline}\label{eq:operator}
\cl_t^{\mu, i_0, v}[{u}](x) := b(t,x,\mu,v,i_0)\dfrac{\partial u}{\partial x}(t, \mu, i_0, x) + \dfrac{1}{2} \si^2(t,x,\mu,v,i_0) \dfrac{\partial^2 u}{\partial x^2}(t,\mu, i_0, x)\\ + \la(t, x, \mu, i_0, v) \int_{\R} \lp u(t, \mu, i_0, x+y) - u(t, \mu, i_0, x) \rp \ga(dy)
\end{multline}
\end{notation}

\noindent
From \eqref{eq:DPP} one obtains the following dynamic programming equation: 
\beq\label{eq:HJB}
-\partial_t V(t, \mu, i_0) + \sup_{v \in A} \ch^v(t, \mu, i_0, {D_m V}) - \sum_{j_0 \in \cs} q_{i_0,j_0} \lp V(t,\mu,j_0) - V(t, \mu, i_0) \rp= 0,
\eeq
where 
\beq\label{eq:ch^v}
\ch^v(t, \mu, i_0, {D_m V}) = -\lla \mu, f(t, \cdot, \mu, v, i_0) + \cl_t^{\mu, i_0, v}[{D_m V}] \rra.
\eeq

\begin{notation}
${\ch}:=\sup_{v \in A}{\ch}^v$.
\end{notation}

\noindent
The following lemma from \cite{soner} is modified for our purposes but its proof is similar.
\begin{lemma}
Under Assumption~\ref{hyp:main} for any $M \in \N$ we have
$$
(t, \mu, i_0) \in \co_M \times \cs \Rightarrow (u, \cl(X_u^{t,\mu,i_0,v} \vert \cf_{u-}^{\al}), \al_u) \in \co_M \times \cs, \quad \forall (u,v) \in [t,T] \times \ca.
$$
\end{lemma}

\subsection{Viscosity solutions and test functions}
The viscosity sub- and super-solutions are defined similar to \cite{soner}. We begin with the class of test functions. 
\begin{definition}
A cylindrical function is a map of the form $(t, \mu,i_0 ) \mapsto  F(t, \langle \mu , f \rangle, i_0) $ for some function $f : \R  \rightarrow  \R$  and $F : [0,T]\times \R \times \cs  \rightarrow  \R$. This function is called cylindrical polynomial if $f$ is a polynomial, and $F(\cdot, \cdot, i_0)$ is continuously differentiable for all $i_0$.
\end{definition}

\begin{definition}
For $E \subseteq  \co$, a viscosity test function on $E \times \cs$ is a function of the form
$$
\varphi (t,\mu,i_0 ) = \sum_{j=1}^{\infty} \varphi_j(t,\mu,i_0), \quad (t,\mu, i_0) \in  E\times \cs,
$$
where $\{ \varphi_j\}_j$ is a sequence of cylindrical polynomials satisfying:
\begin{itemize}
\item[(i)] $\{\phi_j^{i_0} :=\phi_j(\cdot, \cdot, i_0)\}_j$ are absolutely convergent at every $(t, \mu )$ for every $i_0 \in \cs$,
\item[(ii)] for every $j \in  \N$ and $i_0 \in \cs$,
$$
 \lim_{M \to \infty} \sum_{j=M}^{\infty} \sup_{(t,\mu) \in E}\sum_{k=0}^{\textnormal{deg} (D_m \varphi_j^{i_0} )}  \lln \lla \mu ,(D_m \varphi_j^{i_0})^{(k)} \rra \rrn =0.
$$
\end{itemize}
We let $\Phi_{E\times \cs}$ be the set of all viscosity test functions on $E\times \cs$.
\end{definition}

\begin{definition}
For $E \subseteq \bar{\co}$ and $(t, \mu, i_0) \in E \times \cs$ with $t < T$, the superjet of $u$ at $(t, \mu, i_0)$ is given by:
\begin{multline*}
J_{E \times \cs}^{1,+} u(t, \mu, i_0) := \lcl \lp \partial_t \varphi(t, \mu, i_0) , {\lp D_m \varphi(t, \mu, k, \cdot) \rp}_{k \in \cs} \rp \right.\\ 
\left.\bigg\vert \varphi \in \Phi_{E \times \cs}, (u-\varphi)(t, \mu, i_0) = \max_{E}(u-\varphi)(\cdot, \cdot, i_0) \rcl.
\end{multline*}
The subjet of $u$ at $(t, \mu, i_0)$ is defined as $J_{E\times \cs}^{1,-} u(t, \mu, i_0) := - J_{E \times \cs}^{1,+}(-u)(t, \mu, i_0)$.
\end{definition}

\begin{definition}
For a subspace $E \subseteq \bar{\co}$, the upper semi-continuous envelope of $u$ on $E \times \cs$ is defined by
$$
u_{E \times \cs}^{\ast} (t, \mu, i_0) := \limsup_{E \ni (t',\mu') \mapsto (t,\mu)} u(t, \mu, i_0).
$$
The lower semicontinuous envelope $u_{\ast}^{E \times \cs}$ is defined analogously.
\end{definition}

\begin{notation}
$$
u^{\ast} := u_{\bar{\co}\times \cs}^{\ast}, \quad
u_{\ast} := u_{\ast}^{\bar{\co}\times \cs}, \quad
u_M^{\ast} := u_{\bar{\co}_M\times \cs}^{\ast}, \quad
u_{\ast}^M := u_{\ast}^{\bar{\co}_M \times \cs}
$$
\end{notation}

\begin{definition}
We say that a function $u: \co_M\times \cs \to \R$ is a viscosity subsolution of \eqref{eq:HJB} on $\co_M \times \cs$ if for every $(t, \mu, i_0) \in \co_M \times \cs$
$$
-\pi_t + {\ch}(t, \mu, {\pi}_{\mu}) \leq 0, \quad \forall (\pi_t, {\pi}_{\mu}) \in J_{\co_M \times \cs}^{1,+} u_M^{\ast}(t, \mu, i_0).
$$ 

We say that a function $u: \co_M\times \cs \to \R$ is a viscosity supersolution of \eqref{eq:HJB} on $\co_M \times \cs$ if for every $(t, \mu, i_0) \in \co_M \times \cs$
$$
-\pi_t + {\ch}(t, \mu, {\pi}_{\mu}) \geq 0, \quad \forall (\pi_t, {\pi}_{\mu}) \in J_{\co_M \times \cs}^{1,-} u_{\ast}^M(t, \mu, i_0).
$$ 
A viscosity solution of \eqref{eq:HJB} is a function on $\co$ that is both a subsolution and a supersolution of \eqref{eq:HJB} on $\co_M \times \cs$ for every $M \in \N$.
\end{definition}

\begin{proposition}\label{prop:ito-cyl}
For every $\phi \in \Phi_{\co_M \times \cs}$, $(t, \mu, i_0) \in \co_M \times \cs$,  and $v \in \ca$:
\begin{align}\label{eq:ito-cyl}
\phi(u, \mu_u, \al_u) &= \phi(t, \mu, i_0) + \int_t^u \lc \partial_s \phi(s, \mu_s, \al_s) + \lla \mu_s, \cl_s^{\mu_s, \al_s, v}\lc {D_m \phi} \rc \rra\right.\\ &+ \left.\sum_{j_0 \neq \al_{s-}} q_{\al_{s-},j_0} [\phi(s, \mu_s, j_0) - \phi(s, \mu_s, \al_{s-})]  \rc ds\\
 &+ \sum_{j_0 \neq \al_{s-} \in \cs} \int_t^u \lc \phi(s, \mu_s, j_0, \cdot) - \phi(s, \mu_s, i_0, \cdot) \rc dM_{\al_{s-} j_0}^{\al}(s), \quad u \in [t, T],
\end{align}
where $\mu_s = \cl (X_s \vert \cf_{s-}^{\al})$ and $(X_s)_{s \in [t, T]}$ is the solution to \eqref{eq:SDE} with initial conditional distribution $\mu$ and $\al_t=i_0$. 
\end{proposition}

\begin{proof}
For a given polynomial $f$, $\lla \mu_s, f \rra = \E f(X_s)$. Thus using It\^o's formula and taking expectation:
$$
\lla \mu_s, f \rra = \lla \mu, f \rra + \int_t^u \lla \mu_s, \cl_s^{\mu_{s-}, \al_{s-},v}[f] \rra ds.
$$
Now considering a cylindrical polynomial $\phi(t, \mu, i_0) = F(t, \lla \mu, f \rra, i_0)$ we have
\begin{align*}
\phi(u, \mu_u, \al_u) &= \phi(t, \mu, i_0) + \int_t^u \lc \pl_s \phi(s, \mu_s, \al_s) + \pl_x F(s, \lla \mu_s, f \rra) \lla \mu_s, \cl_s^{\mu_{s-}, \al_{s-},v}[f] \rra  \rc ds \\
&+\int_t^u \sum_{j_0 \neq \al_{s-}} q_{\al_{s-},j_0} [\phi(s, \mu_s, j_0) - \phi(s, \mu_s, \al_{s-})]  ds\\
&+ \sum_{j_0 \neq \al_{s-} \in \cs} \int_t^u \lc \phi(s, \mu_s, j_0, \cdot) - \phi(s, \mu_s, i_0, \cdot) \rc dM_{\al_{s-} j_0}^{\al}(s) 
\end{align*}
Since $D_m \phi(s, \mu, i_0) = \pl_x F(s, \lla \mu, f \rra, i_0) f$, we have the result for cylindrical polynomials which can be extended to general $\phi \in \co_M$ as in \cite{soner}.

\end{proof}

\subsection{Value Function}
\begin{lemma}\label{lem:V-bound}
For each $M$, $V$, $V_M^{\ast}$ and $V_{\ast}^M$ are bounded on $\co_M$. 
\end{lemma}
\begin{proof}
The proof is similar to that in \cite{soner}.
\end{proof}

\begin{proposition}
Assume \eqref{eq:DPP} holds. Then for any $M \in \N$, the value function $V$ is both a viscosity sub- and supersolution to \eqref{eq:HJB} on $\co_M$ and
$$
V_M^{\ast}(T, \mu, i_0) = V_{\ast}^M (T, \mu, i_0) = \lla \mu, h(T, \cdot, \mu, i_0) \rra, \forall \mu \in \cm_{M e^{K^{\ast}T}},  i_0 \in \cs. 
$$
\end{proposition}

\begin{proof}
Fix $M \in \N$. By Lemma~\ref{lem:V-bound}, both envelopes $V_M^{\ast}$ and $V_{\ast}^M$ are finite. 

\emph{Step 1: $V_M^{\ast}$ is a viscosity subsolution for $t < T$.} Suppose that for $\phi \in \Phi_{\co_M \times \cs}$ and $(t, \mu) \in \co_M$
$$
0 = (V_M^{\ast} - \phi)(t, \mu, i_0) = \max_{(t,\mu) \in \co_M} \lp V_M^{\ast} - \phi \rp (t, \mu, i_0) \quad \forall i \in \cs.
$$
Fix an arbitrary $i$ and let $(t_n, \mu_n)$ be a sequence in $\co_M$ such that $(t_n, \mu_n, V(t_n, \mu_n, i_0)) \to (t, \mu, V_M^{\ast}(t, \mu, i_0))$. Now fix $v \in A$ and let $(X_s^{t_n, \mu_n, i_0, v})_{s \in [t_n, T]}$ denote the solution to \eqref{eq:SDE} with constant control $v$, $\cf_{-}^{\al}$-conditional distribution $\mu_n$ at time $t_n$ and $\al_{t_n}^n = i_0$. We denote $\mu_s^{n,v} = \cl(X_s^{t_n, \mu_n, i_0, v} \vert \cf_{s-}^{\al})$. Using \eqref{eq:DPP} with $\theta_n = t_n + h$ for $0 < h < T-h_n$, we have
\begin{align*}
V(t_n, \mu_n, i_0) &\leq \E \lc \int_{t_n}^{\theta_n} f(s, X_s^{t_n, \mu_n, i_0, v}, \mu_s^n, v, \al_{s-}^n) ds + V(\theta_n, \mu_{\theta_n}^n, \al_{\theta_n}^n) \rc \\
&\leq  \E \lc \int_{t_n}^{\theta_n} f(s, X_s^{t_n, \mu_n, i_0, v}, \mu_s^n, v, \al_{s-}^n) ds + \phi(\theta_n, \mu_{\theta_n}^n, \al_{\theta_n}^n) \rc.
\end{align*}
Passing to the limit we obtain
$$
V_M^{\ast}(t, \mu, i_0) = \phi(t, \mu, i_0) \leq \E \lc \int_{t}^{t+h} f(s, X_s, \mu_s, v, \al_{s-}) ds + \phi(t+h, \mu_{t+h}, \al_{t+h}) \rc. 
$$
Then using Proposition~\ref{prop:ito-cyl} and recalling that $M_{i_0 j_0}^{\al}$ is a martingale, we have that
$$
0 \leq \E \int_{t}^{t+h} \lc f(s, X_s, \mu_s, v, \al_{s-}) +  \partial_s \phi(s, \mu_s, \al_s) + \lla \mu_s, \cl_s^{\mu_s, \al_s, v} [{D_m \phi}]  \rra \rc ds.
$$
This implies
$$
0 \leq \E^{\al} \int_{t}^{t+h} \lc  \partial_s \phi(s, \mu_s, \al_s) + \lla \mu_s, f(s, \cdot , \mu_s, v, \al_{s-}) + \cl_s^{\mu_s, \al_s, v} [{D_m \phi}] \rra \rc ds.
$$
This holds for any $h > 0$. Recalling $\al_t = i_0$ we then obtain
$$
0 \leq \partial_t \phi(t, \mu, i_0) - {\ch}^v(t, \mu, i_0, {D_m \phi}) 
$$
which implies
$$
-\partial_t \phi(t, \mu, i_0) + {\ch}(t, \mu, i_0, {D_m \phi}) \leq 0. 
$$

\emph{Step 2: $V_{\ast}^M$ is a viscosity supersolution for $t < T$.} Suppose there exists $(t, \mu) \in \co_M$, $\phi \in \Phi_{\co_M \times \cs}$ such that 
$$
0 = (V_{\ast}^M - \phi) (t, \mu, i_0) = \min_{(t, \mu) \in \co_M} (V_{\ast}^M - \phi)(t, \mu, i_0) \quad \forall i \in \cs.
$$
Using \cite[Lemma 7.1]{soner} we get that the above minimum is strict. To derive a contradiction we assume that
$$
-\partial_t \phi(t, \mu, i_0) + {\ch} (t, \mu, i_0, {D_m \phi}) < 0, \quad \text{for some } i \in \cs, \quad \text{for some } i \in \cs.
$$
Since ${\ch}$ is continuous in $(t, \mu)$ there exists a neighborhood $B$ of $(t, \mu)$ such that 
\beq\label{eq:ineq-1}
-\partial_t \phi(t, \mu, i_0) - \lla \mu, f(t, \cdot, \mu, v, i_0) +  \cl_t^{\mu, i_0, v} [{D_m \phi}] \rra < 0, \quad \forall (t, \mu) \in B_M := B \cap \co_M, \forall v \in A.
\eeq
Let $(t_n, \mu_n)$ be a sequence in $\co_M$ such that $(t_n, \mu_n, V(t_n, \mu_n, i_0)) \to (t, \mu, V_{\ast}^M(t, \mu))$. This means that for all large $n$, $(t_n, \mu_n) \in B_M$. Fix an arbitrary control $v \in \ca$ and let \sloppy $(X_s^{t_n, \mu_n, v, i})_{s \in [t_n, T]}$ denote the solution to \eqref{eq:SDE} with $\cf_{-}^{\al}$-conditional distribution $\mu_n$ and value of the Markov chain $\al_{t_n}^n = i$ at initial time $t_n$. Consider
$$
\theta_n := \inf \lcl s \geq t_n : (s, \mu_s^{n,v}) \not\in B_M \text{ or } \al_s^n \neq i \rcl \wedge T.
$$
This is an $\cf^{\al}$ stopping time. By Proposition~\ref{prop:ito-cyl} we have
\begin{multline*}
\phi(t_n, \mu_n, i_0) = \phi(\theta_n, \mu_{\theta_n}^{n,v}, \al_{\theta_n}^n) - \int_{t_n}^{\theta_n} \lc \partial_s \phi(s, \mu_s^n, \al_s^n) + \lla \mu_s^n, \cl_s^{\mu_s^n, \al_{s-}^n, v_s}\lc {D_m \phi} \rc \rra  \rc ds\\
 - \sum_{i_0\neq j_0 \in \cs} \int_{t_n}^{\theta_n} \lla \mu_s^n, D_m \phi(s, \mu_s^n, j_0, \cdot) - D_m \phi(s, \mu_s^n, i_0, \cdot) \rra dM_{i_0 j_0}^{\al^n}(s)
\end{multline*}
From \eqref{eq:ineq-1} we obtain that
\begin{multline*}
\phi(t_n, \mu_n, i_0) = \phi(\theta_n, \mu_{\theta_n}^{n,v}, \al_{\theta_n}) + \int_{t_n}^{\theta_n} \lla \mu_s^n, f(s, \cdot, \mu_s^n, v, \al_{s-}^n) \rra ds\\
 - \sum_{i_0\neq j_0 \in \cs} \int_{t_n}^{\theta_n} \lla \mu_s^n, D_m \phi(s, \mu_s^n, j_0, \cdot) - D_m \phi(s, \mu_s^n, i_0, \cdot) \rra dM_{i_0 j_0}^{\al^n}(s).
\end{multline*}
Since $\overline{\co}_M \setminus B_M = \overline{\co}_M \setminus B$ is compact and $V_{\ast}^M - \phi$ has a strict minimum at $(t, \mu)$, there exists $\eta > 0$ independent of $v$ such that $\phi \leq V_{\ast}^M - \eta \leq V- \eta$ on $\overline{O}_M \setminus B$. Hence we now have
\begin{multline*}
\phi(t_n, \mu_n, i_0) = V(\theta_n, \mu_{\theta_n}^{n,v}, \al_{\theta_n}^n) + \int_{t_n}^{\theta_n} \lla \mu_s^n, f(s, \cdot, \mu_s^n, v_s, \al_{s-}^n) \rra ds\\
 - \sum_{i_0\neq j_0 \in \cs} \int_{t_n}^{\theta_n} \lla \mu_s^n, D_m \phi(s, \mu_s^n, j_0, \cdot) - D_m \phi(s, \mu_s^n, i_0, \cdot) \rra dM_{i_0 j_0}^{\al^n}(s) - \eta.
\end{multline*}
Since $(\phi - V)(t_n, \mu_n, i_0) \to 0$, for $n$ large enough we have
\begin{multline*}
V(t_n, \mu_n, i_0) \leq V(\theta_n, \mu_{\theta_n}^{n,v}, \al_{\theta_n}^n) + \int_{t_n}^{\theta_n} \lla \mu_s^n, f(s, \cdot, \mu_s^n, v_s, \al_{s-}^n) \rra ds\\
 - \sum_{i_0\neq j_0 \in \cs} \int_{t_n}^{\theta_n} \lla \mu_s, D_m \phi(s, \mu_s^n, j_0, \cdot) - D_m \phi(s, \mu_s^n, i_0, \cdot) \rra dM_{i_0 j_0}^{\al^n}(s) - \dfrac{\eta}{2}.
\end{multline*}
Taking expectation we obtain
$$
V(t_n, \mu_n, i_0) \leq \E \lc \int_{t_n}^{\theta_n} \lla \mu_s^n, f(s, \cdot, \mu_s^n, v_s, \al_{s-}^n) \rra ds + V(\theta_n, \mu_{\theta_n}^{n, v}, \al_{\theta_n}^n)  \rc - \dfrac{\eta}{2}
$$
Rewriting the above and noting that $\eta$ is independent of $v \in \ca$ we get
$$
V(t_n, \mu_n, i_0) \leq \E \lc \int_{t_n}^{\theta_n} f(s, X_s^{t_n, \mu_n, i_0, v}, \mu_s^n, v_s, \al_{s-}^n)ds + V(\theta_n, \mu_{\theta_n}^{n, v}, \al_{\theta_n}^n)  \rc - \dfrac{\eta}{2},
$$
contradicting \eqref{eq:DPP}. Hence $V_{\ast}^M$ is a viscosity supersolution to \eqref{eq:HJB}.

\emph{Step 3: $V_M^{\ast}(T, \mu, i_0) = \lla \mu, h(T, \cdot, \mu, i_0) \rra$ for  $\mu \in \cm_{M e^{K^* T}}$.} 
Consider a sequence $(t_n, \mu_n) \in \overline{\co_M}$  converging to $(T, \mu)$ such that $V_M^{\ast} (T, \mu, i_0) = \lim_{n \to \infty} V(t_n, \mu_n, i_0)$. By Assumption~\ref{hyp:main}
$$
\E \lc \int_{t_n}^T \lla \mu_s^n, f(s, \cdot, \mu_s^n, v_s, \al_{s-}^n) \rra ds \rc \to 0,
$$
as $n \to \infty$. By the compactnes of $\overline{\co}_M$, there exists $\hat{\mu} \in \cm_{M}$ such that $\mu_T^{n, v} \to \hat{\mu}$ up to a subsequence. It\^o's formula implies
$$
\lla \mu_T^{n, v} - \mu_n, x^j \rra = \int_{t_n}^{T} \lla \mu_s^n, \cl[x^j] \rra ds \to 0,
$$
owing to \cite[Lemma 6.8]{soner} where $\cl$ is the operator \eqref{eq:operator} acting only on the $x$-variable. This implies $\hat{\mu} = \mu$ courtesy Remark~\ref{rem:mu=nu}. Consequently for arbitrary $v \in \ca$:
\begin{multline*}
V_M^{\ast}(t, \mu, i_0) = \lim_{n \to \infty} V(t_n, \mu_n, i_0)\\ \leq \lim_{n \to \infty} \E \lc \int_{t_n}^T \lla \mu_s^n, f(s, \cdot, \mu_s^n, v_s, \al_{s-}^n) \rra ds + \lla \mu_T^n, h(T, \cdot, \mu_T^n, \al_T^n) \rra \rc = \lla  \mu, h(T, \cdot, \mu_T, i_0)\rra.
\end{multline*}
Since $V_M^{\ast} (T, \mu, i_0) \geq V(T, \mu, i_0) = \lla \mu, h(T, \cdot, \mu, i_0) \rra$, we conclude that $V_M^{\ast}(T, \mu, i_0) = \lla \mu, h(T, \cdot, \mu, i_0) \rra$. 

\emph{Step 4: $V_{\ast}^M (T, \mu, i_0) = \lla \mu, h(T, \cdot, \mu, i_0) \rra$  for $\mu \in \cm_{M e^{K^* T}}$.} 
Consider a sequence $(t_n, \mu_n) \in \overline{\co_M}$  converging to $(T, \mu)$ such that $V_M^{\ast} (T, \mu, i_0) = \lim_{n \to \infty} V(t_n, \mu_n, i_0)$. As before 
$$
\E \lc \int_{t_n}^T \lla \mu_s^n, f(s, \cdot, \mu_s^n, v_s, \al_{s-}^n) \rra ds \rc \to 0,
$$ uniformly in $v \in \ca$. In addition $\lla \mu_T^n, h(T, \cdot, \mu_T^n, \al_T^n) \rra \to \lla \mu, h(T, \cdot, \mu, i_0) \rra$ as $n \to \infty$. For $n \in \N$ choose $v^n \in \ca$ so that
$$
V(t_n, \mu_n, i_0) \geq \E \lc \int_{t_n}^T \lla \mu_s^n, f(s, \cdot, \mu_s^n, v_s^n, \al_{s-}^n) \rra ds + \lla \mu_T^n, h(T, \cdot, \mu_T^n, \al_T^n) \rra \rc - \dfrac{1}{n}.
$$
This implies that 
\begin{multline*}
V_{\ast}^M(T, \mu, i_0) = \lim_{n \to \infty} V(t_n, \mu_n, i_0) \\ \geq \lim_{n \to \infty} \E \lc \int_{t_n}^T  \lla \mu_s^n, f(s, \cdot, \mu_s^n, v_s^n, \al_{s-}^n) \rra ds + \lla \mu_T^n, h(T, \cdot, \mu_T^n, \al_T^n) \rra \rc = \lla \mu, h(T, \cdot, \mu, i_0) \rra.
\end{multline*}
\end{proof}

\subsection{Comparison Result}

\begin{proposition}\label{thm:comparison}
Let $u$ be a u.s.c. subsolution to \eqref{eq:HJB} on ${\co}_M \times \cm$ and $v$ be a l.s.c. subsolution to \eqref{eq:HJB} on ${\co}_M \times \cm$, such that $u(T, \mu, i_0) \leq v(T, \mu, i_0)$ for any $(T, \mu, i_0) \in \overline{\co}_M \times \cm$. Then $u \leq v $ on $\overline{\co}_M \times \cm$.  
\end{proposition}

\begin{proof}
We argue by contradiction. That is, we assume that there exists $i_0 \in \cs$ such that $\sup_{\overline{\co}_M} (u-v)(\cdot, \cdot,i_0) > 0$. Then the proof is similar to \cite[Theorem 8.1]{soner}. We only need a few modifications to factor in the additional argument for the state of the Markov chain and the additional finite sum in the operator $\cl_t^{\mu, v, i}$ arising due to the presence of the the Markov chain, if it appears.   In particular, one can take $u(\cdot, \cdot) \equiv u(\cdot, \cdot, i_0)$ and $v(\cdot, \cdot) \equiv v(\cdot, \cdot, i_0)$ and the proof of \cite[Theorem 8.1]{soner} suffices. 
\end{proof}

\begin{proposition}\label{prop:Lipsct}
The value function $V$ is the unique viscosity solution to \eqref{eq:HJB} on $\co$  satisfying $V^{ \ast} (T, \mu , i_0) = V_{\ast} (T, \mu ,i_0) = \lla \mu, h(T, \cdot, \mu, i_0) \rra$ for $(T, \mu, i_0 ) \in  \co \times \cs$ . Moreover, $V(\cdot, \cdot, i_0)$ restricted to $\co$  is Lipschitz continuous in $\mu$ and $\frac12$-Holder continuous in time.
\beq 
\label{Lip:V}
|V(t,\mu, i_0) - V(s,\nu, i_0)| \leq C \big( d(\mu, \nu) + \sqrt{|t-s|} \big)
\eeq 
where $C$ and $d$ depend on $M$, $\mu \in \mathcal{O}_M$, but $\nu \in \mathcal{O}$.
\end{proposition}
\begin{proof}
The proof of uniqueness follows from the comparison principle. 
Since coefficients are $W_2$-Lipschitz, it is easy to see that $V$ is $W_2$-Lipschitz in $\mathcal{O}$ and so, thanks to the dynamic programming principle, it is $\frac12$-Holder continuous in time. 



To show that  $V$ is Lipschitz in $\mu$ with respect to the distance $d$,  
recall from Assumption \ref{hyp:main} that the coefficients $b$ and $\sigma$ (which do not depend on $x$) and the cost coefficients 
$\mu \rightarrow \langle f (s, \cdot, \mu, v,i), \mu \rangle$ and 
$\mu \rightarrow \langle h (T, \cdot, \mu, i), \mu \rangle$ are Lipschitz with respect to a finite number of moments. 
This assumption gives Lipschitz continuity of the coefficients for $d$. Indeed, since all monomials are contained in $\Theta$ and 
$c(x^n) \leq c( C_{n,k} x^k)$, where $C_{n,k} x^k$ denotes the $n-k$ derivative of $x^n$,  for any $\mu,\nu\in\mathcal{O}$ we have  
\begin{align*}
|g(\mu) - g(\nu)| &\leq \kappa_1 \sum_{k=1}^n |\langle \mu-\nu, x^k \rangle| \leq \kappa_1 \sum_{k=1}^n \frac{1}{C_{n,k}}\frac{c(C_{n,k} x^k)}{c(C_{n,k} x^k)}|\langle \mu-\nu, C_{n,k} x^k \rangle| \\
&\leq \kappa_1 \frac{C_n}{c(x^n)}  \sum_{k=1}^n c(C_{n,k} x^k)|\langle \mu-\nu, C_{n,k} x^k \rangle|
\leq \kappa_1 C_n  d(\mu, \nu) ,
\end{align*}
where $C_n$ denote a constant which depends just on $n$ an $M$. 


To prove Lipschitz-continuity in the measure, fix $i_0\in \mathcal{S}$, $t\in[0,T]$,  $\mu\in \mathcal{O}_M$, consider a control $\alpha$ $\epsilon$-optimal for $(t,\mu)\in \mathcal{O}_M$ and consider another point $\tilde{\mu} \in \mathcal{O}$. Then 
\begin{align*}
V(t,\tilde{\mu}, i_0) &- V(t, \mu,i_0) \leq J(t, \tilde{\mu}, v) 
-J(t, \mu, v) +\epsilon \\ 
& \leq \E \int_t^T  
f( s, \tilde{X}_s, \tilde{\mu}_s, v_s, \alpha_{s-} )
- f( s, X_s, \mu_s, v_s, \alpha_{s-} ) ds 
\\&+ h (T, \tilde{X}_T, \tilde{\mu}_T, \alpha_T ) - 
h (T, \tilde{X}_T, \tilde{\mu}_T, \alpha_T ) +\epsilon\\
&\leq C \sup_{t\leq s \leq T} \E d(\tilde{\mu}_s,  \mu_s) +\epsilon, 
\end{align*}
where $\mu_s = \mathcal{L}(X_s | \alpha_{s-})$, $\tilde{\mu}_s = \mathcal{L}(\tilde{X}_s | \alpha_{s-})$. 
Denoting $\E^\alpha = \E[\cdot |\alpha]$ and $\xi$, $\tilde{\xi}$ such that 
$\mathcal{L}(\xi) =\mu$, $\mathcal{L}(\tilde{\xi}) =\tilde{\mu}$, by It\^o's formula we get (almost surely)
\begin{align*}
	\allowdisplaybreaks
&d(\tilde{\mu}_s,  \mu_s) = \sum_j c_j | \E^\alpha[f_j(\tilde{X}_s) - f_j(X_s)] | \\ 
&\leq \sum_j c_j | \E[f_j(\tilde{\xi}) - f_j(\xi)] | \\
&+ \sum_j c_j \Big| \E^\alpha \int_t^s  \Big\{  f_j'(\tilde{X}_r) b(r, \tilde{\mu}_r, v_r, \alpha_r)  - f_j'(X_r) b(r, \mu_r, v_r, \alpha_r) \\
& \qquad + \frac12 f_j'' (\tilde{X}_r) \sigma^2 (r, \tilde{\mu}_r, v_r, \alpha_r) 
-\frac12 f_j'' (X_r) \sigma^2 (r, \mu_r, v_r, \alpha_r)  \\
& \qquad +\lambda (r, \tilde{\mu}_r, v_r, \alpha_r) \int_\R (f(\tilde{X}_r +y) -f(\tilde{X}_r) ) \gamma(dy)  - \lambda (r, \mu_r, v_r, \alpha_r) \int_\R (f(X_r +y) -f(X_r) ) \gamma(dy) \Big\}dr \Big|.  
\end{align*}
Since 
\[
g_j(x):= \int_\R (f_j(x+y)-f_j(x))\gamma(dy) = \sum_{i=1}^{\mathrm{deg}(f_j)} \frac{f_j^{(i)}(x)}{i!} \int_\R y^i \gamma(dy) =  \sum_{i=1}^{\mathrm{deg}(f_j)} m_i f_j^{(i)}(x),
 \]
the definition of $c_j (b)$ (with $b=M e^{K^*T}$), and the property $\sum_j c_j \langle \mu, f_j \rangle ^2 \leq 1$, as $\mu\in \mathcal{O}_M$, and the boundedness of $b$,  $\sigma$ and $\lambda$ yield 
\begin{align*} 
&	d(\tilde{\mu}_s,  \mu_s) \\
& \leq d(\mu, \tilde{\mu}) + \int_t^s \Big\{ \sum_j c_j |\E^\alpha [f_j'(\tilde{X}_r) -f_j'(X_r) ]| |b|_\infty 
+ C\sum_j c_j |\E^\alpha[f_j'(X_r)] | 
d(\mathcal{L}(\tilde{X}_r),  \mathcal{L}(X_r)) \\
& \qquad + \frac12  \sum_j c_j |\E^\alpha [f_j''(\tilde{X}_r) -f_j''(X_r) ]| |\sigma^2|_\infty 
+ C\sum_j c_j |\E^\alpha[f_j''(X_r)] | d(\mathcal{L}(\tilde{X}_r),  \mathcal{L}(X_r)) \\
&+ \sum_j c_j |\E^\alpha [g_j(\tilde{X}_r) - g_j(X_r) ]| |\lambda|_\infty 
+ C \sum_j c_j |\E^\alpha[g_j(X_r)] | 
d(\mathcal{L}(\tilde{X}_r),  \mathcal{L}(X_r)) 
\Big\} dr \\
& \leq d(\mu, \tilde{\mu}) + C \int_t^s d(\mathcal{L}(\tilde{X}_r),  \mathcal{L}(X_r)) dr, 
\end{align*}
 Thus Gronwall's lemma, taking expectation, gives the claim. 
\end{proof}




\section{Finite agent centralized control problem}\label{sec:finite-agent}
\subsection{Problem setup and viscosity solution definition}
We assume the following state dynamics {interpreted in the weak sense} for an $N-$agent system similar to the controlled McKean-Vlasov jump-diffusion \eqref{eq:SDE}
\beq\label{eq:sde-n}
dX_s^k = b(s, X_s^k, \mu_s^N, v_s^N, \al_{s-}) ds + \si(s, X_s^k, \mu_s^N, v_s^N, \al_{s-})dB_s^k + dJ_s^k,  \quad s > t, \quad k=1, \ldots, N,
\eeq
with initial condition ${X}_s^k = x_k$ and where $\{B^k\}_{k}$ are independent Brownian motions, $\{J^k\}_{k}$ are independent purely discontinuous processes with controlled intensity $\la(s, X_s^k, \mu_s^N, v_s, \al_{s-})$ and the jump sizes are i.i.d. from the distribution $\ga \in \cp(\R)$ satisfying Assumption~\ref{hyp:main}(iii). The Markov chain $\al$ initialized at $\al_t = i_0$ is the noise common to all agents, and is the same as in \eqref{eq:SDE}. {Admissible controls are of the form $v_s^N = v(s, \mu_s^N, \al_{s-})$ taking values in the Polish space $A$. The coefficients $b$, $\si$, $\la$, and the distribution $\ga$ all satisfy the same Assumption~\ref{hyp:main}. The same is true for the running cost $f$ and the terminal cost $h$ in the value function
\beq\label{eq:value-n}
u^N(t, \bx, i_0) = \inf_{v^N \in \ca}   \dfrac{1}{N} \sum_{k=1}^N \E \lc \int_t^T f(s, X_s^k, \mu_s^N, v_s^N, \al_{s-})ds + h(T, X_T^k, \mu_T^N, \al_T) \rc,
\eeq
where $\mathbf{X}_t = \bx \in \R^N$.
The corresponding HJB turns out to be
\begin{multline}\label{eq:HJB-n}
-\dfrac{\pl}{\pl t} u^N(t, \bx, i_0) +  \sup_{v} \dfrac{1}{N} \sum_{k=1}^N H_k^v (t, \bx, i_0, Nu^N, NDu^N, ND^2 u^N)\\ - \sum_{j_0 \in \cs} q_{i_0 j_0} \lc u^N(t, \bx, j_0) - u^N(t, \bx, i_0) \rc = 0, \qquad (t, \bx, i_0) \in [0,T] \times \R^N, \times \cs,
\end{multline}
where 
\begin{multline*}
H_k^v(t, \bx, i_0, u, \ga, \tilde{\ga}) =  \lc-\lc f(t, x_k, \mu^N(\bx), v, i_0) + b(t, x_k, \mu^N(\bx), v, i_0) \ga_i\right.\right.\\ + \left.\left.\dfrac{1}{2} \si^2 (t, x_k, \mu^N(\bx), v, i_0) \tilde{\ga}_{ii} + \la(t, x_k, \mu^N(\bx), v, i_0) \int_{\R} \lc u^N(t, \bx + e_k y, i_0) - u^N(t, \bx, i_0) \rc \ga(dy) \rc \rc,
\end{multline*}
where $\mu^N(\bx) = \dfrac{1}{N} \sum_{k=1}^N \der_{x_k}$. 
Denote 
\beq\label{eq:H^v}
H^{v} := \frac{1}{N}  \sum_{k=1}^N H_k^v
\eeq 
and $H := \sup_v H^v$. 

\begin{definition}
(i) 
A function $u:[0,T] \times \R^N \times \cs \to \R$ is a viscosity subsolution of \eqref{eq:HJB-n} if whenever $\phi \in \cac^{1,2}([0,T] \times \R^N \times \cs)$ and $(u^{\ast} - \phi)(\cdot, \cdot, i_0)$ has a local maxima at $(t, \bx) \in [0,T]\times \R^N$, 
then
$$
-\dfrac{\pl}{\pl t} \phi(t, \bx, i_0) + H(t, \bx, i_0, N\phi, ND\phi, ND^2\phi) + \sum_{j_0 \in \cs} q_{i_0 j_0} \lc u(t, \bx, j_0) - u(t, \bx, i_0) \rc \leq 0.
$$

(ii)
A function $u:[0,T] \times \R^N \times \cs \to \R$ is a viscosity supersolution of \eqref{eq:HJB-n} if whenever $\phi \in \cac^{1,2}([0,T] \times \R^N \times \cs)$ and $(u_{\ast} - \phi)(\cdot, \cdot, i_0)$ has a local minima at $(t, \bx) \in [0,T]\times \R^N$, 
then
$$
-\dfrac{\pl}{\pl t} \phi(t, \bx, i_0) + H(t, \bx, i_0, N\phi, ND\phi, ND^2\phi) + \sum_{j_0 \in \cs} q_{i_0 j_0} \lc u(t, \bx, j_0) - u(t, \bx, i_0) \rc \geq 0.
$$

(iii)
A function $u$ is a viscosity solution of \eqref{eq:HJB-n} if it is both a viscosity subsolution and a viscosity supersolution of \eqref{eq:HJB-n}. 
\end{definition}

\begin{notation}\label{not:emp-proj}
We use the following notation to transform $u^N$ to a function on $\co_M \times \cs$:
\beq\label{eq:hat{u}^N}
\hat{u}^N(t, \mu^N(\bx), i_0) := u^N(t, \bx, i_0), \text{ for } (t, \bx, i_0) \in [0,T] \times \R^N \times \cs.
\eeq
In addition the empirical projection of any $\varphi$ is given by
$$
\tilde{\varphi}^N(t, \bx, i_0) = \varphi(t, \mu^N(\bx), i_0), \text{ for } (t, \bx, i_0) \in [0,T] \times \R^N \times \cs.
$$
Let also 
\[\co^N = \lcl (t, \mu) \in \co : \mu = \frac{1}{N} \sum_{i=1}^N \der_{x_i} \text{ for some } \bx \in \R^N  \rcl
\]
be the set of empirical measures with finite exponential moment, and 
$$
\co_M^N = \lcl (t, \mu) \in \co_M : \mu = \frac{1}{N} \sum_{i=1}^N \der_{x_i} \text{ for some } \bx \in \R^N  \rcl.
$$
\end{notation}





Let us remark that every empirical measure has exponential moments, since it is a finite measure. Moreover, the value function $\hat{u}^N$ is defined on $\co^N$ and not on $\co_M^N$, because the latter set is not invariant for the dynamics of the empirical measure process, while it is invariant for the limiting dynamics.  

\begin{proposition}\label{prop:Was-Lip}
	The value function $u^N$ is the unique viscosity solution to the HJB \eqref{eq:HJB-n}. Furthermore $\hat{u}^N$ is Lipschitz continuous in $\mu\in \mathcal{O}$ for $W_2$ and is $\frac{1}{2}$-H\"older-continuous in $[0,T]$:
	\beq 
	|\hat{u}^N (t, \mu^N(\bd x)) - \hat{u}^N (s, \mu^N(\bd y)) | \leq C
	 \big(  W_2 (\mu^N(\bd x), \mu^N(\bd y)) + |t-s|^{\frac12} \big) .
	 \eeq
\end{proposition}


\begin{proof}
In a setting without regime switches, comparison principle for the viscosity solution of \eqref{eq:HJB-n} is true \cite{pham-viscosity}. Consequently the comparison principle when the regime switches are present follows just as in the Proof of Proposition~\ref{thm:comparison}. Then by an application of the Stochastic Perron's method \cite{bayraktar-li, stoch-perron} we have uniqueness of $u^N$. 

\noindent
Since coefficients are $W_2$-Lipschitz, it is easy to prove that $\hat{u}^N$ is $W_2$-Lipschitz in $\mu$. 
From the dynamic programming principle, it then follows the 1/2 Holder-continuity in time. 
\end{proof}

\subsection{Convergence to mean field control}


\begin{lemma}\label{lem:c}
The following relations hold (see for example \cite{carmona-delarue-volm1})
\begin{align*}
\dfrac{\pl}{\pl x_i} {\phi}^N(t, \bx, i_0) &= \dfrac{1}{N} \dfrac{\pl}{\pl y} D_m \tilde\phi^N(t, \mu^N(\bx), i_0, x_i),\\
\dfrac{\pl^2}{\pl x_i^2} {\phi}^N(t, \bx, i_0) &= \dfrac{1}{N} \dfrac{\pl^2}{\pl y^2} D_m \tilde\phi^N (t, \mu^N(\bx), i_0, x_i) + \dfrac{1}{N^2} \dfrac{\pl}{\pl y'} \dfrac{\pl}{\pl y} D_{m^2}^2 \tilde\phi^N(t, \mu^N(\bx), i_0, x_i, x_i),
\end{align*}
where $\phi^N$ and $\tilde{\phi}^N$ are related by
$$
{\phi}^N(t, \bx, i_0) = \tilde{\phi}^N(t, \mu^N(\bx), i_0).
$$
\end{lemma}

\begin{remark}
In the following we will often use the following distance like quantity instead of $d^2(\mu, \nu)$:
$$
\hat{d}(\mu, \nu) = \sum_{j=1}^{\infty} c_j {\lla \mu - \nu, f_j \rra}^2.
$$
We stress that $d$ depends on $M$, which is the bound on the exponential moments of the measure, because the ${c_j}$ do. 
It is readily checked using Cauchy-Schwarz inequality and the relation $\sum_{j=1}^{\infty} c_j \leq 1$, that
\beq\label{eq:d-dhat}
d^2(\mu, \nu) \leq \hat{d} (\mu, \nu). 
\eeq
\end{remark}


\begin{theorem}\label{thm:1}
	For any $M >0$ we have 
\beq\label{eq:thm1}
\sup_{(t, \mu, i_0) \in \co_M^N \times \cs} \lln V(t, \mu, i_0)- \hat{u}^N(t, \mu, i_0) \rrn \leq \dfrac{C_M}{N^{\frac14}},
\eeq
for a constant $C_M$ depending on $M$. 
\end{theorem}

\begin{remark}
	\label{remark:carda-souga}
	The exponent $1/4$ comes from the $\frac12$-H\"older continuity in time of the value functions, as it is clear from the proofs. Assuming more regularity of the coefficients, we may obtain Lipschitz regularity in time; in the case without jumps and regime switching, such regularity is proved in \cite{carda-souga-1}. Thus, in this case, we can obtain $N^{-\frac12}$ as convergence rate.  
\end{remark}

\begin{proof}
We use the following notations:
\begin{align*}
E_N^{+} &= \sup_{(t, \mu, i_0) \in \co_M^N \times \cs} \lp V(t, \mu, i_0) - \hat{u}^N(t, \mu, i_0) \rp,\\
E_N^{-} &= \sup_{(t, \mu, i_0) \in \co_M^N \times \cs} \lp \hat{u}^N(t, \mu, i_0) - V(t, \mu, i_0) \rp.
\end{align*}
In order to prove \eqref{eq:thm1} it is enough to show that each of $E_N^{+}$ and $E_{N}^{-}$ satisfy the same bound. 
In the following we will only show that 
\beq\label{eq:E-ub-0}
E_N^{+} \leq \frac{C}{N^{\frac14}}
\eeq
and the other case can be done similarly. Note that $E_N^{+}$ can be taken to be positive as otherwise inequality \eqref{eq:E-ub-0} holds trivially.
The proof has been broken into parts for ease in reading. 
We fix $M>0$ and define the distance $d$ according to the constant $b=Me^{K^*} T$.
Please note that in the following, the constants $C$ may depend on $M$ and might change from line to line, but they are not renamed. 

\noindent
\emph{Step 1: Doubling of variables.}
For a positive sequence $(\ep_N, \eta_N) \to (0,0)$, we define the following map on $(\co_M \times \cs)\times (\mathcal{O}^N \times \cs)$, where $\mathcal{P}^N$ is the set of empirical measures $\nu^N(\bd x)$, $\bd x \in \R^N$ : 
\begin{multline}\label{eq:main-a}
\Phi_{N}(t, \mu, i_0, s, \nu, j_0) = V(t, \mu, i_0) - \hat{u}^N (s, \nu, j_0) - \dfrac{1}{2\ep_N} \hat{d}(\mu, \nu) -\dfrac{1}{2 \ep_N} (t-s)^2\\ - \dfrac{1}{2 \ep_N} \tilde d(i_0, j_0) - \dfrac{2T-t-s}{4T} E_N^{+}
- \eta_N \log\Big( 1+\sum_j c_j \lla \nu, f_j \rra^2 \Big),
\end{multline}
where $\hat{u}^N$ is defined above. Note that we add the penalization because $\mathcal{O}^N$ is not compact. 



\noindent
\emph{Step 2: Maximum attained.} 
The maximum of $\Phi_{N}$ is attained and at some $(\bar{t}, \bar{\mu}, \bar{i_0}, \bar{s}, \bar{\nu}, \bar{j_0})$. This is because $\co_M$ is  compact, 
$\lim_{|\bd x| \rightarrow +\infty } \Phi_{N}(t, \mu, i_0, s, \nu^N(\bd x), j_0) = -\infty$
and $\Phi_{N}$ is continuous in $t, \mu, s, \nu$. 

\noindent
\emph{Step 3: Bound on $d(\bar{\mu}, \bar{\nu})$.}
Since $\Phi_{N}(\bar t, \bar \mu, \bar i_0, \bar s, \bar \nu, \bar j_0) \geq \Phi_{N}(\bar t, \bar \nu, \bar i_0, \bar s, \bar \nu, \bar j_0)$, from the Lipschitz continuity of $V$ in \eqref{Lip:V}, we have
$$
\dfrac{1}{2 \ep_N} \hat{d} (\bar \mu, \bar \nu) \leq V(\bar t, \bar \mu, \bar i_0) - V(\bar t, \bar \nu, \bar i_0) \leq C d(\bar \mu, \bar \nu) \leq C \sqrt{\hat d (\bar \mu, \bar \nu)}.
$$
This implies 
$$
\sqrt{\hat d (\bar \mu, \bar \nu)} \leq C \ep_N,
$$
which using relation \eqref{eq:d-dhat} implies 
$$
d(\bar \mu, \bar \nu) \leq C\ep_N.
$$

\noindent
\emph{Step 4: Bound on $|\bar t - \bar s|$.}
Since $\Phi_{N}(\bar t, \bar \mu, \bar i_0, \bar s, \bar \nu, \bar j_0) \geq \Phi_{N}(\bar s, \bar \mu, \bar i_0, \bar s, \bar \nu, \bar j_0)$, we have that
$$
\dfrac{1}{2 \ep_N} (\bar t - \bar s)^2 \leq V(\bar t, \bar \mu, \bar i_0) - V(\bar s, \bar \mu, \bar i_0) + \dfrac{\bar t - \bar s}{4T} E_N^{+} \leq C|\bar t - \bar s|^{\frac12} + \dfrac{\bar t - \bar s}{4T} E_N^{+},
$$
where we have utilized the H\"older property of $V$. Since $E_N^{+}$ is bounded ($V$ and $u^N$ both bounded on $\co_M^N \times \cs$) we obtain
$$
|\bar t - \bar s|^{\frac12} \leq C \ep_N^{\frac13}.
$$

\noindent
\emph{Step 5: Bound on $\tilde d (i_0, j_0)$.}
Since $\Phi_{N}(\bar t, \bar \mu, \bar i_0, \bar s, \bar \nu, \bar j_0) \geq \Phi_{N}(\bar t, \bar \mu, \bar j_0, \bar s, \bar \nu, \bar j_0)$, we have that
$$
\dfrac{1}{2 \ep_N} \tilde d (i_0, j_0) \leq V(\bar t, \bar \mu, \bar i_0) - V(\bar t, \bar \mu, \bar j_0).
$$
Since the right hand side is bounded we have 
$$
\tilde d(i_0, j_0) \leq C \ep_N.
$$
Since $\ep_N \to 0$, for $N$ large enough we have that $\bar i_0 = \bar j_0$. In the sequel we will assume that $N$ is large enough so that this is indeed the case.

\noindent
\emph{Step 6 Case I: $\bar t=T$.} For $(t, \mu, i_0) \in \co_M^N \times \cs$, we have
$\Phi_{N} (\bar t, \bar \mu, \bar i_0, \bar s, \bar \nu, \bar j_0) \geq \Phi_{N} (t, \mu, i_0, t, \mu, i_0)$. This implies
\begin{align*}
&V(t, \mu, i_0) - \hat{u}^N(t, \mu, i_0)\\ 
&\leq \dfrac{2T-2t}{4T} E_N^{+} + V(\bar t, \bar \mu, \bar i_0) - \hat{u}^N(\bar s, \bar \nu, \bar j_0) - \dfrac{1}{2 \ep_N} \hat d(\bar \mu, \bar \nu) - \dfrac{1}{2 \ep_N} (\bar t -\bar s)^2 - \dfrac{\bar t - \bar s}{4 T} E_N^{+}\\
& \qquad -  \eta_N \log\Big( 1+\sum_j c_j \lla \bar{\nu}, f_j \rra^2 \Big) 
+  \eta_N \log\Big( 1+\sum_j c_j \lla \mu, f_j \rra^2 \Big) \\
&\leq \dfrac{1}{2} E_N^{+} + V(\bar t, \bar \mu, \bar i_0) - \hat{u}^N(\bar s, \bar \nu, \bar j_0)
+ \eta_N \log\Big( 1+\sum_j c_j \lla \mu, f_j \rra^2 \Big).
\end{align*}
We notice that
\begin{align*}
V(\bar t, \bar \mu, \bar i_0) - \hat{u}^N(\bar s, \bar \nu, \bar j_0) 
&= \lla \bar \mu, h(T, \cdot, \bar \mu, \bar i_0) \rra - \lla \bar \nu, h(T, \cdot, \bar \nu, \bar i_0) \rra\\
&+ \hat{u}^N(T, \bar \nu, \bar j_0) - \hat{u}^N(\bar s, \bar \nu, \bar j_0).
\end{align*}
Using the Lipschitz assumption on $h$ and Holder continuity of $u^N$ in time, we get
$$
V(\bar t, \bar \mu, \bar i_0) - \hat{u}^N(\bar s, \bar \nu, \bar j_0)  \leq C\lp d(\bar \mu, \bar \nu) 
+ |T-\bar s|^{\frac12} \rp \leq C \ep_N^{\frac13},
$$
as long as $\ep_N \gtrsim \frac{1}{N}$. Now taking supremum over $(t, \mu, i_0)$, and recalling that $\sum_j c_j \lla \mu, f_j \rra^2 \leq 1$ if $\mu\in \mathcal{O}_M$,  we obtain
\beq 
\label{eq:V-lim-0}
E_N^{+} \leq C\big( \ep_N^{\frac13} +\eta_N \big) + \frac{1}{2} E_N^{+}  \implies E_N^{+} \leq C \big( \ep_N^{\frac13} +\eta_N\big). 
\eeq

\noindent
\emph{Step 7 Case II: $\bar s =T$.} Similar to Case I.

\noindent
\emph{Step 7 Case III: $0 \leq \bar t, \bar s < T$.} We use the viscosity solution properties. 

\noindent
(i) $\hat{u}^N - \vp$ has a minimum at $(\bar s, \bar \nu, \bar j_0)$ where
\[
\begin{split}
\vp(s, \nu, j_0) & = V(\bar t, \bar \mu, \bar i_0) - \dfrac{1}{2\ep_N} \hat{d}(\bar \mu, \nu) - \dfrac{1}{2 \ep_N} (\bar t -s)^2 
\\&- \dfrac{1}{2 \ep_N} \tilde d (\bar i_0, j_0) - \dfrac{2T-\bar t -s}{4T} E_N^{+}
-\eta_N \log\Big( 1+\sum_j c_j \lla \nu, f_j \rra^2 \Big).
\end{split}
\]
Recall that $\bar{\nu} = \mu^N(\bx^{\bar{\nu}})$ for some $\bx^{\bar{\nu}} \in \R^n$. Let us now define the function 
$$
\tilde{\vp}^N (s, \bx, j_0) := \vp(s, \mu^N(\bx), j_0).
$$
It is readily checked that $\tilde{\vp}^N \in \cac^{1,2}([0,T] \times \R^N \times \cs)$.
This implies
\begin{multline}\label{eq:ineq-i}
\dfrac{\bar s - \bar t}{\ep_N} - \dfrac{1}{4T} E_N^{+} + \dfrac{1}{N} \sup_v \sum_{k=1}^N H_k^v(\bar s, \bx^{\bar\nu}, \bar j_0, N\tilde{\vp}^N, ND\tilde{\vp}^N, N D^2 \tilde{\vp}^N)\\ + \sum_{j_0 \in \cs} q_{i_0 j_0} \lc \tilde{\vp}^N(\bar s, \bx^{\bar \nu}, j_0) - \tilde{\vp}^N(\bar s, \bx^{\bar \nu}, \bar j_0) \rc \geq 0.
\end{multline}
It is readily checked that 
$$
 \sum_{j_0 \in \cs} q_{i_0 j_0} \lc \tilde{\vp}^N(\bar s, \bx^{\bar \nu}, j_0) - \tilde{\vp}^N(\bar s, \bx^{\bar \nu}, \bar j_0) \rc = \dfrac{1}{2 \ep_N} \sum_{j_0 \in \cs} q_{\bar i_0 j_0} \mathbf{1}_{j_0 \neq \bar j_0}.
$$
Here using Lemma~\ref{lem:c} and the definition of linear derivative
\begin{multline*}
\dfrac{1}{N} \sup_v \sum_{k=1}^N H_k^v(\bar s, \bx^{\nu}, \bar j_0, N\tilde{\vp}^N, ND\tilde{\vp}^N, N D^2 \tilde{\vp}^N) \\
= \dfrac{1}{N} \sup_v \sum_{k=1}^N \lc p(\bar s, x_k^{\bar \nu}, {\bar \nu}, v, \bar j_0, D_m \vp^N) - r(\bar s, x_k^{\bar \nu}, {\bar \nu}, v, \bar j_0, D_m {\vp}) \rc\\
= \sup_v \lla \bar \nu, \lc p(\bar s, \cdot, \bar \nu, v, \bar j_0, D_m \vp^N) - r(\bar s, \cdot, \bar \nu, v, \bar j_0, D_m \vp^N, D_{m^2}^2 \vp^N) \rc \rra
\end{multline*}
where
$$
p(t, x, \nu,  v, i_0, {D_m \phi^N}) = - \lp f(t, x, \nu, v, i_0) + \cl_t^{\nu, i_0, v}[D_m \phi^N](x) \rp 
$$
and 
\begin{multline*}
r(t, x, \nu, v, i_0, {D_m \phi^N}, D_{m^2}^2 \phi^N) = \dfrac{1}{N} \dfrac{\si^2}{2}(t,x,\nu,v,i_0) \dfrac{\pl}{\pl y} \dfrac{\pl}{\pl y'} D_{m^2}^2 \phi^N(t, \nu, i_0, x, x)\\ 
+ \la(t, x, \nu, v, i_0) \lc \int_{\R} \int_0^1 \lcl D_m \phi^N(t, \nu + \dfrac{r}{N}(\der_{x +y} - \der_{x}), i_0, x+y)\right.\right.\\
\left. \left. -  D_m \phi^N(t, \nu + \dfrac{r}{N}(\der_{x +y} - \der_{x}), i_0, x)\rcl dr \right. \\
\left. -\lp  D_m \phi^N(t, \nu, i_0, x+y) - D_m \phi^N (t, \nu, i_0, x) \rp \rc \ga(dy)
\end{multline*}


\noindent
(ii) $V-\psi$ has a maximum at $(\bar t, \bar \mu, \bar i_0)$ where 
\[
\begin{split}
\psi(t, \mu, i_0) &= \hat{u}^N(\bar s, \bar \nu, \bar j_0) + \dfrac{1}{2 \ep_N} d^2(\mu, \bar \nu) + \dfrac{1}{2 \ep_N} (t - \bar s)^2 + \dfrac{1}{2 \ep_N} \tilde d(i_0, \bar j_0) 
\\&+ \dfrac{2T-\bar s -t}{4T} E_N^{+} + \eta_N \log\Big( 1+\sum_j c_j \lla \bar \nu, f_j \rra^2 \Big).
\end{split}
\]
This implies 
$$
-\pl_t \psi(\bar t, \bar \mu, \bar i_0) + H(\bar t, \bar \mu, \bar i_0, D_m \psi) + \sum_{j_0 \in \cs} q_{\bar i_0 j_0} \lc \psi(\bar t, \bar \mu, j_0) - \psi(\bar t, \bar \mu, \bar i_0) \rc \leq 0.
$$
Consequently
\beq\label{eq:ineq-ii}
-\dfrac{\bar t - \bar s}{\ep_N} + \dfrac{1}{4T}E_N^{+} + H(\bar t, \bar \mu, \bar i_0, D_m \psi) + \sum_{j_0 \in \cs} q_{\bar i_0 j_0} \lc \psi(\bar t, \bar \mu, j_0) - \psi(\bar t, \bar \mu, \bar i_0) \rc \leq 0.
\eeq
Here it is readily checked that
$$
\sum_{j_0 \in \cs} q_{\bar i_0 j_0} \lc \psi(\bar t, \bar \mu, j_0) - \psi(\bar t, \bar \mu, \bar i_0) \rc = \dfrac{1}{2 \ep_N} \sum_{j_0 \in \cs} q_{\bar i_0 j_0} \mathbf{1}_{\{j_0 \neq \bar i_0\}}.
$$
In addition
we have
\begin{multline*}
H(\bar t, \bar \mu, \bar i_0, D_m \psi) =  \sup_v \lla \bar \mu, - \lcl f(\bar t, \cdot, \bar \mu, v, \bar i_0) + \cl_{\bar t}^{\bar \mu, \bar i_0, v} [D_m \psi] \rcl \rra\\
 = \sup_v \lla \bar \mu, p(\bar t, \cdot, \bar \mu, v, \bar i_0, D_m \psi) \rra.
\end{multline*}

\noindent
(iii) Adding \eqref{eq:ineq-i} and \eqref{eq:ineq-ii} we obtain that
\begin{align}\label{eq:E-ub}
&\dfrac{1}{2T} E_N^{+} \leq  \sup_v\lla \bar \nu, ~ \lc p(\bar s, \cdot, \bar \nu, v, \bar j_0, D_m \vp^N) - r(\bar s, \cdot, \bar \nu, v, \bar j_0, D_m \vp^N) \rc \rra \nonumber\\
& - \sup_v \lla \bar \mu, p(\bar t, \cdot, \bar \mu, v, \bar i_0, D_m \psi) \rra \nonumber \\
&\leq \sup_v \lla \bar \nu, \lcl p(\bar s, \cdot, \bar \nu, v, \bar j_0, D_m \vp^N) - p(\bar t, \cdot, \bar \mu, v, \bar i_0, D_m \psi) \rcl \rra - \inf_v \lla \bar \mu - \bar \nu, p(\bar t, \cdot, \bar \mu, v, \bar i_0, D_m \psi) \rra \nonumber \\
& - \inf_v \lla \bar \nu,  r(\bar s, \cdot, \bar \nu, v, \bar j_0, D_m \vp^N) \rra \nonumber \\
&=: I_1 + I_2 + I_3.
\end{align}


\noindent
Assumption~\ref{hyp:main} gives that the function
$\lla  \mu, f(\bar t, \cdot, \mu, v, \bar i_0)\rra$ is Lipschitz for $d$. Thus 
\[
\lla \bar \mu, f(\bar t, \cdot, \bar \mu, v, \bar i_0)\rra - \lla \bar \nu, f(\bar s, \cdot, \bar \nu, v, \bar i_0)\rra
\leq C(|\bar t- \bar s| + d (\bar \mu, \bar \nu)) \leq C \ep_N^{\frac23}.
\]

\noindent
\emph{Bound for $I_1$.} 
We have
\beq\label{eq:Dm-vp}
D_m \vp(t, \nu, j_0, x) = \dfrac{1}{\ep_N} \sum_{j=1}^{\infty} c_j \lla \bar \mu - \nu , f_j\rra f_j(x) 
- 2 \eta_N \frac{\sum_j c_j \lla \mu, f_j \rra f_j(x)}{1+\sum_j c_j \lla \mu, f_j \rra^2 },
\eeq
\beq\label{eq:Dm-psi}
D_m \psi(t, \mu, i_0, x) = \dfrac{1}{\ep_N} \sum_{j=1}^{\infty} c_j \lla \mu - \bar \nu , f_j\rra f_j(x),
\eeq
and, since $\bar\mu\in \mathcal{O}_M$ 
\beq 
\label{eq:424}
\sum_j c_j \lla \bar \nu, f_j \rra^2 \leq 2 \hat{d}(\bar \mu, \bar \nu) 
+ 2 \sum_j c_j \lla \bar \mu, f_j \rra^2 \leq 2C \ep_N^2 +2 \leq 3  
\eeq
if $N$ is large enough. 
We compute
\begin{align*}
& \lln b(\bar s,  \bar \nu, v, \bar i_0) \dfrac{\pl}{\pl x}D_m \vp^N(\bar s, \bar \nu, \bar i_0, x) - b(\bar t,  \bar \mu, v, \bar j_0) \dfrac{\pl }{\pl x}D_m \psi(\bar t, \bar \mu, \bar i_0, x)\rrn\\
& \leq \lln b(\bar s,  \bar \nu, v, \bar i_0) - b(\bar t,  \bar \mu, v, \bar j_0) \rrn \dfrac{1}{\ep_N} \sum_{j=1}^{\infty} c_j \lln \lla \bar \mu - \bar \nu, f_j \rra f_j'(x) \rrn 
+ |b(\bar s,  \bar \nu, v, \bar i_0)| 2 \eta_N \frac{\sum_j c_j |\lla \bar\nu, f_j \rra f_j'(x)|}{1+\sum_j c_j \lla \bar\nu, f_j \rra^2 }
\\
& \leq C(|\bar t -\bar s| + d(\bar \mu, \bar \nu)) \dfrac{1}{\ep_N} \sum_{j=1}^{\infty} c_j \lln \lla \bar \mu - \bar \nu, f_j \rra f_j'(x) \rrn 
+C\eta_N  \frac{\sum_j c_j |\lla \bar\nu, f_j \rra f_j'(x)|}{1+\sum_j c_j \lla \bar\nu, f_j \rra^2 }\\
&\leq \frac{C}{\ep_N^{\frac13}} \sum_{j=1}^{\infty} c_j \lln \lla \bar \mu - \bar \nu, f_j \rra  f_j'(x)\rrn +C\eta_N  \frac{\sum_j c_j |\lla \bar\nu, f_j \rra f_j'(x)|}{1+\sum_j c_j \lla \bar\nu, f_j \rra^2 }.
\end{align*}
Similarly we obtain that
\begin{align*}
&\lln \dfrac{\si^2}{2}(\bar s,  \bar \nu, v, \bar i_0) \dfrac{\pl^2}{\pl x^2}D_m \vp^N(\bar s, \bar \nu, \bar i_0, x) - \dfrac{\si^2}{2}(\bar t, x, \bar \mu, v, \bar j_0) \dfrac{\pl^2 }{\pl x^2}D_m \psi(\bar t, \bar \mu, \bar i_0, x)\rrn\\
& \leq \frac{C}{\ep_N^{\frac13}} \sum_{j=1}^{\infty} c_j \lln \lla \bar \mu - \bar \nu, f_j \rra f_j''(x) \rrn
+C\eta_N  \frac{\sum_j c_j |\lla \bar\nu, f_j \rra f_j''(x)|}{1+\sum_j c_j \lla \bar\nu, f_j \rra^2 },
\end{align*}
and 
\begin{align*}
&\lln \lambda(\bar s, x, \bar \nu, v, \bar i_0) \int_{\R} \lp D_m \vp^N(\bar s, \bar \nu, \bar i_0, x+y) - D_m \vp^N(\bar s, \bar \nu, \bar i_0, x) \rp \ga(dy)\right.\\
&\left.  - \lambda(\bar t, x, \bar \mu, v, \bar i_0) \int_{\R} \lp D_m \psi (\bar t, \bar \mu, \bar i_0, x+y) - D_m \psi(\bar t, \bar \mu, \bar i_0, x) \rp \ga(dy) \rrn\\
& \leq \frac{C}{\ep_N^{\frac13}}  \sum_{j=1}^{\infty} c_j \lln \lla \bar \mu - \bar \nu, f_j \rra  \int_{\R} \lp f_j(x+y) - f_j(x) \rp \ga(dy) \rrn 
+C\eta_N  \frac{\sum_j c_j |\lla \bar\nu, f_j \rra \int_{\R} \lp f_j(x+y) - f_j(x) \rp \ga(dy) |}{1+\sum_j c_j \lla \bar\nu, f_j \rra^2 }
\\
&= \frac{C}{\ep_N^{\frac13}} \sum_{j=1}^{\infty} c_j \lln \lla \bar \mu - \bar \nu, f_j \rra g_j(x) \rrn
+C\eta_N  \frac{\sum_j c_j |\lla \bar\nu, f_j \rra g_j(x) |}{1+\sum_j c_j \lla \bar\nu, f_j \rra^2 },
\end{align*}
where the last equality follows by Taylor's expansion and the polynomial 
$$
g_j(x) = \int_{\R} \lp f_j(x+y) - f_j(x) \rp \ga(dy) =\sum_{i=1}^{\infty} m_i f_j^{(i)}(x).
$$
Consequently
\begin{align*}
&\lln p(\bar s, x, \bar \nu, v, \bar j_0, D_m \vp^N) - p(\bar t, x, \bar \mu, v, \bar i_0, D_m \psi) \rrn\\
 &\leq \frac{C}{\ep_N^{\frac13}}   \sum_{j=1}^{\infty} c_j \lln \lla \bar \mu - \bar \nu, f_j \rra  \lp f_j'(x) + f_j''(x) + g_j(x) \rp \rrn 
+C\eta_N  \frac{\sum_j c_j |\lla \bar\nu, f_j \rra \lp f_j'(x) + f_j''(x) + g_j(x) \rp|}{1+\sum_j c_j \lla \bar\nu, f_j \rra^2 } .
\end{align*}
This implies by Cauchy-Schwarz inequality
\begin{align}\label{eq:V-lim-a}
\lln I_1 \rrn &\leq \frac{C}{\ep_N^{\frac13}} {\lp \sum_{j=1}^{\infty} c_j {\lla  \bar \mu - \bar \nu, f_j \rra}^2   \sum_{j=1}^{\infty} c_j \lp {\lla \bar \nu, f_j' \rra}^2 + {\lla \bar \nu, f_j'' \rra}^2 + {\lla \bar \nu, g_j \rra}^2 \rp \rp}^{1/2} \nonumber\\
&\qquad +C\eta_N  \frac{\lp  \sum_{j=1}^{\infty} c_j { \lla\bar \nu, f_j \rra}^2   \sum_{j=1}^{\infty} c_j \lp {\lla \bar \nu, f_j' \rra}^2 + {\lla \bar \nu, f_j'' \rra}^2 + {\lla \bar \nu, g_j \rra}^2 \rp \rp^{1/2}}{1+\sum_j c_j \lla \bar\nu, f_j \rra^2}
\nonumber\\ 
&\leq  \frac{C}{\ep_N^{\frac13}} (\hat{d}(\bar \mu, \bar \nu))^{1/2} +C\eta_N
 \leq C \ep_N^{\frac23} + C \eta_N,
\end{align}
since $\hat{d}(\bar \mu, \bar \nu) \leq C \ep_N^2$ and $\sum_{j=1}^{\infty} c_j {\lla \bar \nu, \phi_j \rra}^2 \leq 3 $ for $\phi_j = f_j', f_j'', g_j$ (see \cite{soner} and \eqref{eq:424}).

\noindent
\emph{Bound for $I_2$.} By Assumption~\ref{hyp:main} and similar computations as above, we have
\begin{align*}
&b(\bar t, x, \bar \mu, v, \bar i_0) \dfrac{\pl}{\pl x} D_m \psi (\bar t, \bar \mu, \bar i_0, x) \leq \dfrac{C}{\ep_N} \sum_{j=1}^{\infty} c_j \lla \bar \mu - \bar \nu, f_j \rra f_j'(x), \\
&\dfrac{\si^2}{2}(\bar t, x, \bar \mu, v, \bar i_0) \dfrac{\pl^2}{\pl x^2} D_m \psi (\bar t, \bar \mu, \bar i_0, x) \leq \dfrac{C}{\ep_N} \sum_{j=1}^{\infty} c_j \lla \bar \mu - \bar \nu, f_j \rra f_j''(x),\\
&\la(\bar t, x, \bar \mu, v, \bar i_0) \int_{\R} \lp  D_m \psi (\bar t, \bar \mu, \bar i_0, x+y) -  D_m \psi (\bar t, \bar \mu, \bar i_0, x) \rcl \ga(dy) \leq \dfrac{C}{\ep_N} \sum_{j=1}^{\infty} c_j \lla \bar \mu - \bar \nu, f_j \rra g_j(x).
\end{align*}
This implies 
\begin{align*}
I_2&= \sup_{v \in A} \lla \bar \mu- \bar \nu ,  \cl_{\bar t}^{\bar \mu, \bar i_0, v} [D_m \psi] \rra  \\
& = \frac{1}{\ep_N}  \sum_{j=1}^{\infty} c_j \lla \bar \mu - \bar \nu, f_j \rra 
\lc \lla \bar \mu - \bar \nu, b(\bar t,  \bar \mu, v, \bar i_0) f_j' \rra + \lla \bar \mu - \bar \nu, \dfrac{\si^2}{2}(\bar t, \bar \mu, v, \bar i_0) f_j'' \rra + \lla \bar \mu - \bar \nu, \la(\bar t, \bar \mu, v, \bar i_0) g_j \rra \rc
\\
&\leq \dfrac{C}{\ep_N} \sum_{j=1}^{\infty} c_j \left|\lla \bar \mu - \bar \nu, f_j \rra \lc \lla \bar \mu - \bar \nu, f_j' \rra + \lla \bar \mu - \bar \nu, f_j'' \rra + \lla \bar \mu - \bar \nu, g_j \rra \rc\right|\\
&\leq \dfrac{C}{\ep_N} {\lp \sum_{j=1}^{\infty} c_j {\lla \bar \mu - \bar \nu, f_j \rra}^2 \rp}^{1/2} \lp {\lp  \sum_{j=1}^{\infty} c_j {\lla \bar \mu - \bar \nu, f_j' \rra}^2 \rp}^{1/2} + {\lp  \sum_{j=1}^{\infty} c_j {\lla \bar \mu - \bar \nu, f_j'' \rra}^2 \rp}^{1/2}\right. \\ 
&+ \left.{\lp  \sum_{j=1}^{\infty} c_j {\lla \bar \mu - \bar \nu, g_j \rra}^2 \rp}^{1/2}\rp \leq \dfrac{C}{\ep_N} \hat{d}(\bar \mu, \bar \nu) \leq C \ep_N.
\end{align*}



\noindent
Consequently we have
\beq\label{eq:V-lim-b}
|I_2| \leq C \ep_N.
\eeq

\noindent
\emph{Bound for $I_3$.} Observe that 

\begin{align*}
D_{m^2}^2 \vp^N &(\bar s, \bar \nu, \bar j_0, y, y') 
= \dfrac{1}{\ep_N} \sum_{j=1}^{\infty} c_j  f_j(y) f_j(y')  \\
&- 2\eta_N \frac{ \sum_{j=1}^{\infty} c_j  f_j(y) f_j(y') 
	\big(1+ \sum_{j=1}^{\infty} c_j \lla  \bar \nu, f_j \rra^2 \big) 
	-2  \sum_{j=1}^{\infty} c_j \lla  \bar \nu, f_j \rra f_j(y) \sum_{i=1}^{\infty} c_i \lla  \bar \nu, f_i \rra f_i(y') }{\big(1+ \sum_{j=1}^{\infty} c_j \lla  \bar \nu, f_j \rra^2 \big)^2} \\
& = r^N_1 +\eta_N r^N_2
\end{align*}


We  bound the second term as 
\begin{align*}
&\lla \bar \nu,  \dfrac{\pl}{\pl y} \dfrac{\pl}{\pl y'} r^N_2 (\bar\nu ,x,x) \rra  \\
&\qquad\leq 2\frac{ \sum_{j=1}^{\infty} c_j  \lla \bar \nu , f_j'^2  \rra
	\big(1+ \sum_{j=1}^{\infty} c_j \lla  \bar \nu, f_j \rra^2 \big) 
	+ 2  \big(\sum_{j=1}^{\infty} c_j  \lla  \bar \nu, f_j \rra \lla \bar \nu , f_j'^2  \rra ^{\frac12} \big)^2 }{\big(1+ \sum_{j=1}^{\infty} c_j \lla  \bar \nu, f_j \rra^2 \big)^2} \\
&\qquad\leq 2\frac{ \sum_{j=1}^{\infty} c_j  \lla \bar \nu , f_j  \rra
	\big(1+ 3\sum_{j=1}^{\infty} c_j \lla  \bar \nu, f_j \rra^2 \big) }{\big(1+ \sum_{j=1}^{\infty} c_j \lla  \bar \nu, f_j \rra^2 \big)^2} \\
&\leq 6 \frac{ \Big(\sum_{j=1}^{\infty} c_j \lla  \bar \nu, f_j \rra^2 \Big)^{\frac12} }{1+ \sum_{j=1}^{\infty} c_j \lla  \bar \nu, f_j \rra^2 } \leq 6,
\end{align*}
while the first term is bounded by 
\begin{align*}
\lla \bar \nu, \dfrac{1}{N} r^N_1 (\bar\nu,x,x) \rra &= \dfrac{1}{N \ep_N} \sum_{j=1}^{\infty} c_j  \lla \bar \nu, (f_j')^2 \rra\\
&\leq \dfrac{1}{N \ep_N} \sum_{j=1}^{\infty} c_j \lla \bar \nu, f_j \rra \\
&\leq  \dfrac{C}{N \ep_N}
\end{align*}
where we have used \eqref{eq:424}.
 Furthermore we notice that 
\begin{align*}
I_3^1(x,y) &= \int_0^1 \lc D_m \vp^N(\bar s, \bar \nu+\dfrac{r}{N} (\der_{x+y} - \der_x), x+y) - D_m \vp^N(\bar s, \bar \nu, \bar j_0, x+y) \rc dr \\
&= \int_0^1 \dfrac{1}{\ep_N} \sum_{j=1}^{\infty} c_j \lp \lla \bar \mu - (\bar \nu + \dfrac{r}{N}(\der_{x+y} - \der_{x})), f_j \rra - \lla \bar \mu - \bar \nu, f_j \rra \rp f_j(x+y) dr\\
&= \dfrac{1}{2N\ep_N} \sum_{j=1}^{\infty} c_j \lp f_j(x+y) - f_j(x) \rp f_j(x+y).
\end{align*}
Let us omit the penalization term with the log, which is bounded in a similar way.  Similarly
\begin{align*}
I_3^2(x,y) &= \int_0^1 \lc D_m \vp^N(\bar s, \bar \nu+\dfrac{r}{N} (\der_{x+y} - \der_x), x) - D_m \vp^N(\bar s, \bar \nu, \bar j_0, x) \rc dr \\
&= \dfrac{1}{2N\ep_N} \sum_{j=1}^{\infty} c_j \lp f_j(x+y) - f_j(x) \rp f_j(x).
\end{align*}
This implies 
\begin{align*}
\int_{\R} &\lc I_3^1(x,y) - I_3^2(x,y) \rc \ga(dy) =\dfrac{1}{2N\ep_N} \sum_{j=1}^{\infty} c_j \int_{\R} {\lcl f_j(x+y) - f_j(x) \rcl}^2 \ga(dy)\\
&\leq \dfrac{1}{2N\ep_N} \sum_{j=1}^{\infty}c_j {\lp \int_{\R} (f_j(x+y) - f_j(x)) \ga (dy) \rp}^2 = \dfrac{1}{2N\ep_N} \sum_{j=1}^{\infty} c_j g_j^2(x). 
\end{align*}
Thus
$$
\lla \bar \nu, \int_{\R} [I_3^1(\cdot,y) - I_3^2(\cdot,y)] \ga(dy) \rra \leq \dfrac{1}{2N\ep_N} \sum_{j=1}^{\infty} c_j \lla \bar \nu, g_j^2 \rra \leq \dfrac{1}{2N\ep_N} \sum_{j=1}^{\infty} c_j {\lla \bar \nu, g_j \rra}^2 \leq \dfrac{1}{2N\ep_N}.
$$
Since $\la$ and $\si$ are bounded we conclude that
\beq\label{eq:V-lim-c}
|I_3| \leq \dfrac{C}{N \ep_N} + C\eta_N . 
\eeq

\noindent
\emph{Bound for $E_N^{+}$.} Combining \eqref{eq:V-lim-0}, \eqref{eq:V-lim-a}, \eqref{eq:V-lim-b} and \eqref{eq:V-lim-c}, we obtain from \eqref{eq:E-ub} 
\[
E_N^+ \leq  C \Big( \ep_N^{\frac13} +\eta_N + \frac{1}{N\ep_N}   \Big).
\]
By choosing the optimal $\ep_N = \frac{1}{N^{\frac34}}$ and $\eta_N$ arbitrarily, we get the   desired result
$ E_N^{+} \leq \dfrac{C}{N^{\frac14}}$.
\end{proof}

\section{Propagation of chaos}\label{sec:poc}

In this section we consider, for completeness, the case in which the value function of the limiting mean field control problem is smooth. In such case, the limiting optimal trajectory is unique and thus our aim in to prove a quantitative propagation of chaos result, that is, to bound the distance between the optimal trajectories $\mu_t^{N}$ and $\mu_t$ for the finite-agent and mean field control problems respectively. Notably, this bound involves not only the value functions, but also their  measure derivatives. The bound obtained in this section is more standard and  independent of the previous main result Theorem \ref{thm:1}, and is obtained in terms of the Wasserstein distance. 


To that aim we introduce some additional assumptions. Note that here we run our analylis on the entire $\co^N$, not on $\co^N_M$. 

\begin{assumption}\label{hyp:poc}
We have that $V(t, \mu, i_0)$ is continuously differentiable in time and twice continuously (linear functional) differentiable in the measure $\mu$. The derivative $D_m \vp$ is twice continuously differentiable and $D_{m^2}^2 \vp$ once continuously differentiable in space. All the derivatives are uniformly bounded. Further, $u^N \in C^{1,2}([0,T]\times \R^N)$. 
\end{assumption}

This assumption permits to improve the convergence rate for the value functions in 
 \eqref{eq:thm1}, while to bound the optimal trajectories
we also need some assumptions on the pre-Hamiltonian $\ch^{v}$ given by \eqref{eq:ch^v}.
\begin{assumption}\label{hyp:poc-H}
The pre-Hamiltonian $\ch^v(t, \mu, i_0, D_m V)$ is continuously differentiable in $v \in A$, $\pl_v \ch^v$ is Lipschitz continuous with respect to $\mu$, and $\ch^v$ is uniformly concave in $A$:
$$
\ch^{v_2}(t, \mu, i_0, D_mV) \leq \ch^{v_1} (t, \mu, i_0, D_mV) - \pl_v \ch^{v_1} (t, \mu, i_0, D_m V) |v_1 - v_2| - \la |v_1 - v_2|^2, 
$$
for some $\la>0$.
\end{assumption}

Let $v^{\ast,N}$ be the optimal feedback control for the $N$-agent optimization and $v^{\ast}$ be the optimal feedback control for the the mean field control problem. The optimal controls are unique because of the assumed regularity and uniform concavity. 

\begin{remark}
Since $v \mapsto \ch^v$ is continuously differentiable and strictly concave, $v^{\ast}(t, \mu, i_0)$ appears as the unique solution of the equation $\pl_{v} \ch^v (t, \mu, i_0, D_mV) = 0$. By the implicit function theorem, $v^{\ast}$ is also Lipschitz continuous in the space and measure arguments. 
\end{remark}

\begin{remark}
In a setting without jumps and when the diffusion coefficient $\si$ is uncontrolled, Assumption~\ref{hyp:poc-H} can be seen in \cite{carmona-delarue-p-mfg} where the drift coefficient is assumed to be an affine function of the control and the cost function is assumed to be convex. 
\end{remark}

Observe that from \eqref{eq:H^v} using Lemma~\ref{lem:c} and denoting $\vp_k(t, \bx) \equiv \vp(t, x_k, \mu^N(\bx), v, i_0)$ for $\vp = f$, $b$, $\si$, $\la$ we have
\begin{align*}
&H^v(t, \bx, i_0, Nu^N, NDu^N, ND^2u^N)
= -\dfrac{1}{N} \sum_{k=1}^N \lc {f_k(t, \bx)} + b_k(t, \bx) \dfrac{\pl}{\pl y} D_m \hat{u}^N(t, \mu^N(\bx), i_0, x_k) \right.\\
&\left.+ \dfrac{\si_k^2}{2}(t, \bx) \lp \dfrac{\pl^2}{\pl y^2} D_m \hat{u}^N (t, \mu^N(\bx), i_0, x_k) + \dfrac{1}{N} \dfrac{\pl}{\pl y} \dfrac{\pl}{\pl y'} D_{m^2}^2 \hat{u}^N(t, \mu^N(\bx), i_0, x_k, x_k) \rp\right.\\
&\left. + \la_k(t, \bx) \int_{\R} \int_0^1 \lc D_m \hat{u}^N(t, \mu^N(\bx) + \dfrac{r}{N} (\der_{x_k+y} - \der_{x_k}), i_0, x_k+y\right.\right.)\\
& \hspace{2.2in}  \left.\left.- D_m\hat{u}^N(t, \mu^N(\bx) + \dfrac{r}{N}(\der_{x_k + y} - \der_{x_k}), i_0, x_k) \rc ds \ga(dy) \rc,
\end{align*}
where $\hat{u}^N$ is given by \eqref{eq:hat{u}^N}.
Recalling \eqref{eq:ch^v} we thus have
\begin{multline}\label{eq:H-relation}
H^v(t, \bx, i_0, Nu^N, NDu^N, ND^2u^N)= \ch^v(t, \mu^N(\bx), i_0, D_m \hat{u}^N) + R_1^v(t,\mu^N(\bx), i_0, \hat{u}^N)\\ + R_2^v(t,\mu^N(\bx), i_0, \hat{u}^N), 
\end{multline}
where 
$$
R_1^v(t, \mu^N(\bx), i_0, \hat{u}^N) = \dfrac{1}{N}\lla \mu^N(\bx), \dfrac{\pl}{\pl y} \dfrac{\pl}{\pl y'} D_{m^2}^2 \hat{u}^N(t, \mu^N(\bx), i_0, \cdot, \cdot) \rra,
$$
and
\begin{multline*}
R_2^v(t, \mu^N(\bx), i_0, \hat{u}^N) = \lla \mu^N(\bx) , \la(s, \cdot, \mu^N(\bx), v, i_0) \lc \int_{\R} \bigg[D_m \hat{u}^N(t, \mu^N(\bx), i_0, \cdot +y)\right.\right.\\
\left.\left. - \int_0^1 D_m \hat{u}^N(t, \mu^N(\bx) + \dfrac{r}{N}(\der_{\cdot + y} - \der_{\cdot}), i_0, \cdot+y) dr \bigg] \ga(dy)- \int_{\R}\bigg[ D_m \hat{u}^N(t, \mu^N(\bx), i_0, \cdot) \right.\right.\\
\left.\left. - \int_0^1 D_m \hat{u}^N(t, \mu^N(\bx) + \dfrac{r}{N}(\der_{\cdot + y} - \der_{\cdot}), i_0, \cdot)dr\bigg]\ga(dy) \rc \rra
\end{multline*}

\begin{lemma}\label{lem:R}
Under Assumption~\ref{hyp:poc} for $\mu^N(\bx) \in \co_M^N$
$$
|R_1^{v^{\ast,N}}(s, \mu^N(\bx), i_0, \hat{u}^N)| + |R_2^{v^{\ast,N}}(s, \mu^N(\bx), i_0, \hat{u}^N)| \leq \dfrac{C}{N}.
$$
As a consequence, $\tilde{V}^N$ almost solves \eqref{eq:HJB-n}:
\begin{multline}\label{eq:HJB-VN}
	-\dfrac{\pl}{\pl t} \tilde{V}^N(t, \bx, i_0) +  \sup_{v} \dfrac{1}{N} \sum_{k=1}^N H_k^v (t, \bx, i_0, N\tilde{V}^N, ND\tilde{V}^N, ND^2 \tilde{V}^N)\\ - \sum_{j_0 \in \cs} q_{i_0 j_0} \lc \tilde{V}^N(t, \bx, j_0) - \tilde{V}^N(t, \bx, i_0) \rc = \mathcal{O}\Big(\frac1N\Big), \qquad (t, \bx, i_0) \in [0,T] \times \R^N, \times \cs,
\end{multline}
and we obtain 
\beq\label{eq:V1N}
\sup_{(t, \mu, i_0) \in \co^N \times \cs} \lln V(t, \mu, i_0)- \hat{u}^N(t, \mu, i_0) \rrn \leq \dfrac{C}{N}. 
\eeq
\end{lemma}

\begin{proof}
The term $R_1^{v^{\ast, N}} $ is bounded 
because $D_{m^2}^2$ is continuously differentiable 
with uniformly bounded derivatives.   
Observe that by definition of the linear functional derivative (with non-relevant arguments omitted)
\begin{align*}
&D_m \hat{u}^N(\mu^N(\bx) + \dfrac{r}{N}(\der_{x+y} - \der_{y}), x+y) - D_m \hat{u}^N(\mu^N(\bx), x+y)\\
 &= \dfrac{r}{N} \int_0^1 \bigg( D_{m^2}^2\hat{u}^N(\mu^N(\bx) + \dfrac{r_1 r}{N} (\der_{x+y} - \der_x), x+y, x+y)\\
 &\hspace{0.75in}- D_{m^2}^2 \hat{u}^N(\mu^N(\bx) + \dfrac{r_1 r}{N} (\der_{x+y} - \der_x),x+y,y) \bigg) dr_1.
\end{align*}
Since $\ga$ satisfies Assumption~\ref{hyp:main}(iii), 
 we conclude that $R_2^{v^{\ast, N}} \leq \frac{C}{N}$. This implies \eqref{eq:HJB-VN} and, as a consequence, $\tilde{V}^N(t,\bx,i_0) = V(t, \mu^N(\bx),i_0)$ is such that 
 $\tilde{V} - \frac{C}{N}(T-t)$ is a (classical) subsolution to \eqref{eq:HJB-n}, while $\tilde{V} + \frac{C}{N}(T-t)$ is a (classical) supersolution. Thus the comparison principle gives 
 \[
 \sup_{(t, \bx, i_0) \in [0,T] \times \R^N \times  \cs} \lln \tilde{V}(t, \bx, i_0)- u^N(t, \bx, i_0) \rrn \leq \dfrac{C}{N},
 \]
 which is equivalent to \eqref{eq:V1N}.
\end{proof}

We can now estimate the distance between the feedback functions

\begin{lemma}
Let $\mu^N_s$ be the optimal empirical measure process for the centralized finite-agent control problem. Under Hypotheses~\ref{hyp:poc} and \ref{hyp:poc-H} we have
\beq\label{eq:v-bound}
\E \int_0^T |v^{\ast}(s, \mu^N_s, \alpha_s) - v^{\ast,N}(s, \mu^N_s, \alpha_s)|^2 ds \leq \dfrac{C}{N}.
\eeq
\end{lemma}

\begin{proof}
We compute the expected value of $V$ along $\mu^N$.
\begin{align*}
&\E \lc V(T, \mu_T^{N}, \al_T) \rc - \E\lc V(0, \mu_0^{N}, \al_0) \rc = \E \lc \int_0^T \dfrac{\pl}{\pl s} V(s, \mu_s^N, \al_s) ds \rc\\
& = \E \lc -\int_0^T \ch^{v^{\ast}} (s, \mu_s^N, \al_s, D_mV)ds + \int_0^T \sum_{j_0 \in \cs} q_{\al_{s-}j_0} [V(s, \mu_s^N, j_0) - V(s, \mu_s^N, \al_{s-})] \rc.
\end{align*}
From \eqref{eq:V1N} we thus have
\begin{align*}
 &\E \lc V(T, \mu_T^{N}, \al_T) \rc - \E\lc V(0, \mu_0^{N}, \al_0) \rc\\
 &\leq  \E\lc - \int_0^T \ch^{v^{\ast,N}}(s, \mu_s^N, \al_s, D_m V) +\sum_{j_0 \in \cs} q_{\al_{s-} j_0} [\hat{u}^N(s, \mu_s^N, j_0) - \hat{u}^N(s, \mu_s^N, \al_{s-})] ds \rc\\
&\hspace{1in}+ \E \lc \int_0^T \lp \ch^{v^{\ast, N}} (s, \mu_s^N, \al_s, D_m V) - \ch^{v^{\ast}}(s, \mu_s^N, \al_s, D_mV) \rp  ds \rc + \dfrac{C}{N}.
\end{align*}
From the relation \eqref{eq:H-relation} we then have
\begin{align*}
&\E \lc V(T, \mu_T^{N}, \al_T) \rc - \E\lc V(0, \mu_0^{N}, \al_0) \rc\\
&\leq \E \lc -\int_0^T H^{v^{\ast, N}} (s, \mu_s^N, \al_s, N\tilde{V}^N, ND\tilde{V}^N, ND^2\tilde{V}^N)ds\right.\\
& \left.+\sum_{j_0 \in \cs} q_{\al_{s-} j_0} [\hat{u}^N(s, \mu_s^N, j_0) - \hat{u}^N(s, \mu_s^N, \al_{s-})] ds \rc
+ \int_0^T [R_1^{v^{\ast,N}}(s, \mu_s^N, \al_s, \tilde{V}^N) + R_2^{v^{\ast,N}}(s, \mu_s^N, \al_s, \tilde{V}^N)]  ds  \\
&+\E \lc \int_0^T \lp \ch^{v^{\ast, N}} (s, \mu_s^N, \al_s, D_m V) - \ch^{v^{\ast}}(s, \mu_s^N, \al_s, D_mV) \rp  ds \rc + \dfrac{C}{N}.
\end{align*}
Consequently since $R_1 + R_2$ satisfy Lemma~\ref{lem:R} this implies from \eqref{eq:HJB-n}
\begin{align*}
\E \lc V(T, \mu_T^{N}, \al_T) \rc - & \E\lc V(0, \mu_0^{N}, \al_0) \rc \leq \E \lc \hat{u}^N(T, \mu_T^{N}, \al_T) \rc - \E\lc \hat{u}^N(0, \mu_0^{N}, \al_0) \rc\\ 
&+ \E \lc \int_0^T \lp \ch^{v^{\ast, N}} (s, \mu_s^N, \al_s, D_m V) - \ch^{v^{\ast}}(s, \mu_s^N, \al_s, D_mV) \rp  ds \rc + \dfrac{C}{N}.
\end{align*}
Again using Theorem~\ref{thm:1} we now obtain
\beq\label{eq:5.2-1}
\E \lc \int_0^T \lp \ch^{v^{\ast}} (s, \mu_s^N, \al_s, D_mV) - \ch^{v^{\ast, N}}(s, \mu_s^N, \al_s, D_mV) \rp  ds \rc \leq  \dfrac{C}{N}.
\eeq
Since $v^{\ast}$ maximizes the pre-Hamiltonian $\ch^{v}(s, \mu_s^N, \al_s, D_mV)$ we have from Assumption~\ref{hyp:poc-H}
$$
\ch^{v^{\ast}} (s, \mu_s^N, \al_s, D_m V) - \ch^{v^{\ast,N}}(s, \mu_s^N, \al_s, D_mV) \geq \la |v_s^{\ast} - v_s^{\ast,N}|^2.
$$
This implies from \eqref{eq:5.2-1} that \eqref{eq:v-bound} holds. 
\end{proof}


Let $\cb = \{B^k\}_{k=1}^N$ be independent Brownian motions, $\cu = \{U_l\}_{l=1}^{\infty}$ be an infinite sequence of i.i.d. random variables from $\ga$, and $\cn = \{N^k\}_{k=1}^N$ be independent standard Poisson processes. Recall $v^{\ast,N}$ is the unique optimal feedback control for the $N$-agent optimization and $\mu^{ N}$ is the corresponding optimal measure trajectory. In addition let $(X^{\ast,k})_{1\leq k \leq N}$ be the corresponding optimal jump diffusion with $v^{\ast, N}$ as the control and driven by $\cb$, $\cu$ and $\cn$ as follows 
$$
dX_s^{\ast,k} = b(s, X_s^{\ast,k}, \mu_s^N, v_s^{\ast,N}, \al_{s-}) ds + \si(s, X_s^{\ast,k}, \mu_s^{N}, v_s^{\ast,N}, \al_{s-}) dB_s^k + dJ_s^{\ast,k},
$$
where 
$$
J_s^{\ast,k} = \sum_{l=1}^{N_t^{\ast,k}} U_l, \qquad N_t^{\ast,k} = N^k \lp \int_0^t \la \lp s, X_s^{\ast,k},\mu_s^N, v_s^{\ast,N}, \al_{s-} \rp ds\rp.
$$
Let $(Y^k)_{1 \leq k \leq N}$ be defined similarly driven by the same $\cb$, $\cu$ and $\cn$, but with $v^{\ast}$ as the control. Let $\rho^N$ be the corresponding empirical measure trajectory. 

Recall $v^{\ast}$ is the unique optimal feedback control for the the mean field control problem and $\mu$ is the corresponding optimal measure trajectory. In addition let $X^{\ast}$ be the corresponding optimal jump diffusion with $v^{\ast}$ as the control. We also consider $N$ i.i.d. copies of $X^{\ast}$, namely, $\{\tilde{X}^{\ast,k}\}_{1 \leq k \leq N}$ which are driven by $\cb$, $\cu$ and $\cn$.}


The proof of the following theorem is a standard argument of propagation of chaos, thus it is omitted. It uses the $W_2$-Lipschitz-continuity of the coefficients and the established bound \eqref{eq:v-bound}. 

\begin{theorem}\label{thm:5-1}
Under Assumptions~\ref{hyp:poc}, \ref{hyp:poc-H}, and \ref{hyp:main}-(v), 
\begin{align}
	\E \lp \sup_{s \in [0,T]} \frac{1}{N} \sum_{k=1}^N |X_s^{\ast,k} - Y_s^k|^2 \rp &\leq \dfrac{C}{N}\\
	\E \lp \sup_{s \in [0,t]} \frac{1}{N} \sum_{k=1}^N |\tilde{X}_s^{k} - \tilde{X}_s^{\ast} |^2 \rp &\leq \dfrac{C}{N} \\
	\E \lc \sup_{t \in [0,T]} W_2(\mu_t^{N}, \rho^N_t) \rc &\leq \dfrac{C}{\sqrt{N}} \\
	\E \lc \sup_{t \in [0,T]} W_2(\mu_t^{N}, \mu_t) \rc &\leq \dfrac{C}{N^{1/9}} 	
\end{align}
\end{theorem}

\bigskip

\bibliographystyle{plain}
\bibliography{references}

\end{document}